\documentclass{article}

\usepackage{amsfonts}
\usepackage{amsmath}
\usepackage{mathtools}
\usepackage{multirow}
\usepackage{graphicx}
\usepackage{epstopdf}
\usepackage{algorithm}
\usepackage{algorithmic}
\usepackage{bm}
\usepackage{microtype}
\usepackage{amssymb}
\usepackage{enumerate}
\usepackage{color}
\usepackage{marvosym}
\usepackage{subcaption}
\usepackage{hyperref}
\usepackage{url}
\usepackage{makecell}

\newtheorem{theorem}{Theorem}

\newtheorem{remark}{Remark}
\newtheorem{lemma}{Lemma}
\newtheorem{proof}{Proof}

\newtheorem{definition}{Definition}
\newtheorem{assumption}{Assumption}

\newcommand{\0}{\mathbf{0}}

\newcommand{\E}{\mathbb{E}}
\newcommand{\bO}{{\cal O}}
\newcommand{\R}{\mathbb{R}}

\newcommand{\F}{\mathcal{F}}

\newcommand{\I}{\mathbf{I}}

\newcommand{\V}{\mathbf{V}}

\renewcommand{\L}{\mathbf{L}}
\newcommand{\M}{\mathbf{M}}

\newcommand{\G}{\mathbf{G}}

\newcommand{\fO}{\mathbf{D}}
\newcommand{\X}{\mathbf{X}}
\newcommand{\A}{\mathbf{A}}
\newcommand{\Y}{\mathbf{Y}}
\newcommand{\Z}{\mathbf{Z}}

\newcommand{\bL}{\mathbf{L}}
\newcommand{\bR}{\mathbf{R}}
\newcommand{\bP}{\mathbf{P}}
\newcommand{\bQ}{\mathbf{Q}}
\newcommand{\bPh}{\mathbf{\Phi}}

\newcommand{\<}{\left\langle}
\renewcommand{\>}{\right\rangle}
\DeclareMathOperator*{\argmax}{argmax}

\usepackage[accepted]{icml2020}

\icmltitlerunning{Convergence of Rotation-based Matrix Optimizers}

\begin{document}

\onecolumn
\icmltitle{Convergence of Rotation-based Matrix Optimizers: \\A Unified Analysis of SOAP, Conda, and SPlus}

\begin{icmlauthorlist}
\icmlauthor{Yiwen Sun}{gooo}
\icmlauthor{Huan Li}{to}
\icmlauthor{Zhouchen Lin}{goo}
\end{icmlauthorlist}

\icmlaffiliation{to}{Institute of Robotics and Automatic Information Systems, College of Artificial Intelligence, Nankai University, Tianjin, China.\\}
\icmlaffiliation{goo}{National Key Lab of General AI, School of Intelligence Science and Technology, Peking University, Beijing, China.\\}
\icmlaffiliation{gooo}{School of the Gifted Young, University of Science and Technology of China, Hefei, China.\\}
\icmlcorrespondingauthor{Huan Li and Zhouchen Lin}{lihuanss@nankai.edu.cn, zlin@pku.edu.cn}





\vskip 0.3in



\printAffiliationsAndNotice{}  

\begin{abstract}
  In this work, we develop a unified theoretical framework for analyzing the convergence of rotation-based matrix optimizers, which apply orthogonal transformations to map the momentum into a rotated space, perform coordinate-wise or normalized updates there, and then rotate the updates back. Our framework encompasses prominent matrix optimizers, including SOAP, Conda, and truncated SPlus, as well as matrix-parameterized Adam as a special case. As our main result, we establish, for the first time, the convergence rate of SOAP with sharp dimensional dependence, as well as the convergence rate of Adam measured by the nuclear norm. Our framework also covers a more general variant of rotation-based optimizers that allows arbitrary orthogonal rotation matrices, allowing these matrices to depend on the current stochastic gradient at each iteration. Technically, our framework leverages a row-column trace-control argument that converts elementwise bounds into bounds on two diagonal control matrices, thereby improving the dimension dependence of the resulting bound.
\end{abstract}

\section{Introduction}

Pre-training large language models incurs substantial computational costs. Over the past decade, the Adam family of algorithms, including AdaGrad \citep{Duchi-2011-jmlr,McMahan-2010-colt}, RMSProp \citep{RMSProp-2012-hinton}, Adam \citep{adam-15-iclr}, and AdamW \citep{adanw-2019-iclr}, has served as the de facto standard for optimizing deep neural networks. These methods treat the network’s parameters as a single high-dimensional vector and do not explicitly exploit row/column correlations within matrices. More recently, matrix-based optimizers, which exploit the inherent matrix structure of deep neural network parameters, have garnered increasing attention and demonstrated the potential to outperform their vector-based counterparts. Prominent examples include Shampoo \citep{shampoo-icml-18}, Muon \citep{muon2024}, SOAP \citep{soap-2025-iclr}, SPlus \citep{splus-2025-nips}, and Conda \citep{Conda}. The latter three approaches can be collectively classified as rotation-based optimizers, which rotate the gradient or momentum into a new space via an orthogonal transformation, produce an update in that space, and then rotate the update back to the original space. These methods differ mainly in how the rotation matrices and the transformed update are constructed: SOAP combines the eigenbasis of Shampoo’s preconditioners, used as the rotation matrix, with an Adam-style update; SPlus replaces the Adam-style update with a sign-type update; and Conda employs the singular-vector basis of the momentum as the rotation matrix. A general algorithmic formulation encompassing most popular algorithms in this family is provided in \citep{aro-gong-26}. While \citet{aro-gong-26} provide a unified framework at the algorithmic level, our work provides a unified framework at the level of theoretical analysis.

Although the convergence behavior of vector-based optimizers has been thoroughly studied in the literature \citep{bottou-2022-tmlr,luo-2020-iclr,lihuan-rmsprop-2024,luo-2022-nips,hong-2024-adam,haochuanli-2023,Li-2025-nips}, the theoretical understanding of matrix-based optimizers remains comparatively underdeveloped. To date, convergence and convergence rates have been established for only a few matrix optimizers, such as Muon \citep{muon-hongmingyi-25,Muon-NS-ICLR-26,muon-shen-25,muon-chen-25,muon-sato-25} and Shampoo \citep{shampoo-li-26,shampoo-xie-25}. For rotation-based matrix optimizers, most notably SOAP, rigorous convergence analysis is still lacking. Relative to Muon and Shampoo, the main technical hurdles stem from two intrinsic properties of rotation-based updates. First, it is non-trivial to combine the properties in the two spaces; moreover, the rotation matrices may depend on the current stochastic gradient. Second, coordinate-wise updates in the rotated space introduce an undesirable dependence on the number of matrix entries, $mn$.

\subsection{Contributions}\label{sec:contribution}
To address the above issues, we treat SOAP, Conda, and SPlus in a unified framework with the following contributions:
\begin{enumerate}
\item We establish the following convergence rate
\begin{eqnarray}
\begin{aligned}\notag
	\frac{1}{K}\hspace*{-0.05cm}\sum_{k=1}^K\E\left[\left\|\nabla f(\X_k)\right\|_*\right]\hspace*{-0.05cm}\leq\hspace*{-0.05cm} \bO\hspace*{-0.07cm}\left(\hspace*{-0.07cm}\sqrt{m\hspace*{-0.05cm}+\hspace*{-0.05cm}n}\max\left\{\sqrt[4]{\frac{\sigma^2L\left(f(\X_1)\hspace*{-0.05cm}-\hspace*{-0.05cm}f^*\right)}{K}},\sqrt{\frac{L\left(f(\X_1)\hspace*{-0.05cm}-\hspace*{-0.05cm}f^*\right)}{K}}\right\}\right)
\end{aligned}
\end{eqnarray}
for generalized SOAP, Conda and the truncated modification of SPlus, matching the state-of-the-art rate of Shampoo \citep{shampoo-li-26} for matrix optimizers. 
\item Our analysis extends to a considerably more general variant that accommodates arbitrary orthogonal rotation matrices, requiring only that they be orthogonal. This flexibility allows the rotation matrices to be constructed adaptively from the entire optimization history up to the current step (and thus may depend on the current stochastic gradient), covering those used in SOAP, Conda, and SPlus. Consequently, the theoretical guarantees we establish not only apply directly to these methods but also shed light on a broader class of rotation-based matrix optimizers, providing a unified analytical framework for understanding their convergence behavior.

\item In particular, our analysis also applies when the rotation matrices reduce to the identity matrix, in which case SOAP reduces to Adam and thus the above convergence rate also applies to Adam. By contrast, the existing literature measures the convergence rate of Adam only in the $\ell_1$ or $\ell_2$ norm; we provide the first nuclear-norm convergence rate for Adam with matrix-valued parameters, identical to that of matrix optimizers such as AdamW-style Shampoo \citep{shampoo-li-26}.

\item Technically, we propose a row-column trace-control argument in which the control matrices are carefully designed, reducing the problem to rigorously bounding their traces and lowering the dimensional dependence from $\sqrt{mn}$ to $\sqrt{m+n}$. To handle the difficulty posed by orthogonal transformations between the two spaces, we exploit unitary invariance, which holds for any unitarily invariant norm such as the nuclear norm and the Frobenius norm.
\end{enumerate}

\subsection{Problem Settings, Notations, and Assumptions}\label{sec:assumption}
We study the following nonconvex problem with matrix-valued parameters in this paper
\begin{equation}\label{problem}
\min_{\X\in\mathbb R^{m\times n}} f(\X),
\end{equation}
where $f(\X)=\E_{\zeta\sim\mathcal{P}}[f(\X;\zeta)]$ and $\zeta$ is the sample drawn from the data distribution $\mathcal{P}$. We assume that 
$f$ is bounded below and denote $f^*=\inf_{\X} f(\X)$.

We denote vectors by lowercase bold letters and matrices by uppercase bold letters. For vectors,
denote $\|\cdot\|_2$ as the $\ell_2$ Euclidean norm. For matrices, denote $\|\cdot\|_F$, $\|\cdot\|_{op}$, and $\|\cdot\|_*$ as the Frobenius norm, spectral norm (largest singular value), and nuclear norm (sum of singular values), respectively. Denote $\X_{k,i,j}$, $\X_{k,i,:}$, and $\X_{k,:,j}$ to be the $(i,j)$-th element, $i$-th row, and $j$-th column of matrix $\X$ at iteration $k$, respectively. For matrices $\X$ and $\Y$, denote $\frac{\X}{\Y}$, $\X^2$, $|\X|$, and $\X^{\frac{1}{p}}$ as the elementwise division, square, absolute value, and the $\frac{1}{p}$-th power (different from the matrix power, with a slight abuse of terminology), respectively. Denote $\F_k=\sigma(\zeta_1,\zeta_2,\cdots,\zeta_k)$ to be the sigma field of the stochastic gradients up to $k$, denote $\E_{\F_k}[\cdot]$ as the expectation with respect to $\F_k$ and $\E_k[\cdot|\F_{k-1}]$ the conditional expectation with respect to $\zeta_k$ given $\F_{k-1}$. For the sake of brevity, $\E_{\F_K}[\cdot]$ will be denoted as $\E[\cdot]$. Denote $\sigma_i(\A)$ and $\lambda_i(\A)$ as the $i$-th singular value and eigenvalue of $\A$, respectively.

We make the following assumptions throughout this paper:
\begin{assumption}
	Smoothness: $\|\nabla f(\Y)-\nabla f(\X)\|_F\leq L\|\Y-\X\|_F, \forall \X,\Y$.
\end{assumption}
\begin{assumption}
	Unbiased estimator: $\E_k\left[\G_k\big|\F_{k-1}\right]=\nabla f(\X_k)$.
\end{assumption}
\begin{assumption}
	Bounded noise variance: $\E_k\left[\|\G_k-\nabla f(\X_k)\|_F^2\big|\F_{k-1}\right]\leq\sigma^2$.
\end{assumption}

\subsection{Related Work}
Among matrix-based methods, the convergence properties of Muon have been extensively investigated and are now largely understood, owing to its relatively simple structure \citep{muon-hongmingyi-25,muon-shen-25,muon-chen-25,muon-sato-25}. For more intricate optimizers, recent progress includes the work of \citet{shampoo-li-26}, who established rigorous convergence guarantees for a practical AdamW-style implementation of Shampoo. Nevertheless, with the exception of these two cases, the convergence analysis of other matrix-based optimizers remains an open problem to the best of our knowledge, especially for the representative SOAP optimizer. Perhaps the most closely related study is that of \citet{galore-icml-25}, which analyzes a variant of the GaLore algorithm \citep{galore-icml-24}. However, their approach replaces the orthogonal rotation employed by SOAP and semi-orthogonal low-rank projection by GaLore with a random projection matrix and uses momentum stochastic gradient descent (MSGD) within the projected space, whereas SOAP and GaLore leverage the adaptive mechanisms of Adam. Consequently, the analytical techniques developed in \citep{galore-icml-25} are far from sufficient to establish the convergence of the full SOAP algorithm, as well as of Conda, and SPlus. \citet{galore-icml-25} also constructed a counterexample demonstrating the non-convergence of GaLore due to its low-rank projection. Other works have studied SOAP from different perspectives \citep{KL-26-iclr,soap-25}, but none has provided a convergence guarantee.

\begin{algorithm}[t]
	\caption{General Rotation-Based Matrix Optimizer}
	\label{general}
	\begin{algorithmic}
		\STATE Hyper parameters: $\eta,\theta$. \STATE Initialize $\X_1$, $\M_0=\0$.
		\FOR{$k=1,2,\cdots,K$}
		\STATE $\G_k=\mbox{GradOracle}(\X_k)$
        \STATE $\M_k=\theta\M_{k-1}+(1-\theta)\G_k$
        \STATE $\mbox{Generate arbitrary orthogonal matrices }\bP_k\in\mathbb R^{m\times m}\mbox{ and }\bQ_k\in\mathbb R^{n\times n}$
		\STATE $\X_{k+1}=\X_k-\eta\bP_k\bPh_k\bQ_k^\top$, where $\bPh_k= \frac{\bP_k^\top\M_k\bQ_k}{\fO_{k}}$ for some $\fO_k$
		\ENDFOR
	\end{algorithmic}
\end{algorithm}

\begin{table*}
\caption{Representative Rotation-Based Optimizers with Their Rotation Matrices and Updates
}\label{table-comp}
\begin{center}
\renewcommand{\arraystretch}{1.3}
\begin{tabular}{|c|c|c|}
\hline
 Optimizers & Rotation Matrices     & Update in the rotated space  \\
 \hline
 SOAP &  Eigenbasis of Shampoo’s preconditioner & Adam-style update \\
 SPlus & Eigenbasis of Shampoo’s preconditioner & Sign-type update\\
 Conda & Singular-vector basis of the momentum & Adam-style update\\
 \hline
 \makecell{Our \\ generalizations}
  & Arbitrary orthogonal matrices & \makecell{Adam-style or \\ truncated sign-type update}\\
\hline
\end{tabular}
\end{center}
\end{table*}

\section{Rotation-based Matrix Optimizers}\label{uni-frame} 
We begin by examining the general framework of rotation-based matrix optimizers presented in Algorithm \ref{general}. Here $\bPh_k$ denotes the update method used in the rotated space, such as the Adam-style update in SOAP and Conda, and the sign-type update in SPlus. We require the update $\bPh_k$ to admit the following representation for an entrywise positive matrix $\fO_{k}$:
\begin{equation}\notag
\bPh_k= \frac{\bP_k^\top\M_k\bQ_k}{\fO_{k}},
\end{equation}
where $\fO_k$ can be regarded as a normalizer. We describe the choice of $\fO_k$ for each algorithm individually in the following subsections. The orthogonal rotation matrices $\bP_k$ and $\bQ_k$ may be generated by any procedure, including periodic updates, and need not be independent of the current stochastic gradient at each iteration. We also provide examples of the generation process for each algorithm in the following subsections. The rotation matrices and the transformed update capture the main differences among these algorithms, thus allowing the framework to represent and include many popular algorithms. Table \ref{table-comp} lists the representative examples studied in this paper.

\newpage
\vspace*{-1.2cm}\noindent\begin{minipage}[t]{0.5\textwidth}
	\begin{algorithm}[H]
		\caption{SOAP \citep{soap-2025-iclr}}
		\label{soap}
		\begin{algorithmic}
			\STATE Hyper parameters: $\eta,\theta,\beta,\varepsilon,T$
			\STATE Initialize $\X_1$, $\M_0$, $\V_0$, $\bL_0$, $\bR_0$, $\bP_1$, $\bQ_1$.
			\FOR{$k=1,2,\cdots,K$}
			\STATE $\G_k=\mbox{GradOracle}(\X_k)$
			\STATE $\M_k=\theta\M_{k-1}+(1-\theta)\G_k$
			\STATE $\G_k'=\bP_k^\top\G_k\bQ_k$
			\STATE $\M_k'=\bP_k^\top\M_k\bQ_k$
			\STATE $\V_k=\beta\V_{k-1}+(1-\beta)(\G_k')^2$
			\STATE $\X_{k+1}=\X_k-\eta\bP_k\frac{\M_k'}{\sqrt{\V_k+\varepsilon}}\bQ_k^\top$
			\STATE $\bL_k=\beta \bL_{k-1}+(1-\beta)\G_k\G_k^\top$
			\STATE $\bR_k=\beta \bR_{k-1}+(1-\beta)\G_k^\top\G_k$
			\IF{$k\mod T=0$}
			\STATE $\bP_{k+1}=\mbox{Eigenvector}(\bL_k)$
			\STATE $\bQ_{k+1}=\mbox{Eigenvector}(\bR_k)$
			\ELSE
			\STATE $\bP_{k+1}=\bP_k$, $\bQ_{k+1}=\bQ_k$
			\ENDIF
			\ENDFOR
		\end{algorithmic}
	\end{algorithm}
    \vspace*{-0.7cm}\begin{algorithm}[H]
	\caption{SPlus \citep{splus-2025-nips}}
    \label{splus}
	\begin{algorithmic}
		\STATE Hyper parameters:$\eta, \theta, \beta, \kappa, T$
		\STATE Initialize $\X_1$, $\X_1^{\prime}$, $\M_0$, $\L_0$, $\bR_0$, $\bP_0$, $\bQ_0$.
		\FOR{$k=1,2,\cdots,K$}
		\STATE $\G_k=\mbox{GradOracle}(\X_k^{\prime})$
		\STATE $\M_k=\theta\M_{k-1}+(1-\theta)\G_k$
        \STATE $\bL_k=(1-\beta)\bL_{k-1}+\beta \G_k\G_k^\top$
        \STATE $\bR_k=(1-\beta)\bR_{k-1}+\beta \G_k^\top\G_k$
        \IF{$k\mod T=0$}
		\STATE $\bP_{k}=\mbox{Eigenvector}(\bL_k)$
		\STATE $\bQ_{k}=\mbox{Eigenvector}(\bR_k)$
		\ELSE
		\STATE $\bP_{k}=\bP_{k-1}$, $\bQ_{k}=\bQ_{k-1}$
		\ENDIF
		\STATE $\X_{k+1}^{\prime}=\!\X_k^{\prime}\!-\!\frac{2\eta}{m+n}\bP_k\!\operatorname{Sign}(\bP_k^\top\M_k\bQ_k)\bQ_k^\top$
        \STATE $\X_{k+1}=(1-\kappa)\X_{k}+\kappa\X_{k+1}^{\prime}$
		\ENDFOR
	\end{algorithmic}
\end{algorithm}
\end{minipage}%
\hfill
\begin{minipage}[t]{0.48\textwidth}
	\begin{algorithm}[H]
		\caption{Generalized SOAP}
		\label{gsoap}
		\begin{algorithmic}
			\STATE Hyper parameters: $\eta,\theta,\beta,\varepsilon$
			\STATE Initialize $\X_1$, $\M_0=\0$, $\V_0=\0$.
			\FOR{$k=1,2,\cdots,K$}
			\STATE $\G_k=\mbox{GradOracle}(\X_k)$
			\STATE $\M_k=\theta\M_{k-1}+(1-\theta)\G_k$
            \STATE $\mbox{Generate arbitrary orthogonal matrices }$
            \STATE $\bP_k\in\mathbb R^{m\times m}\mbox{ and }\bQ_k\in\mathbb R^{n\times n}$
			\STATE $\V_k=\beta\V_{k-1}+(1-\beta)(\bP_k^\top\G_k\bQ_k)^2$
			\STATE $\X_{k+1}=\X_k-\eta\bP_k\frac{\bP_k^\top\M_k\bQ_k}{\sqrt{\V_k+\varepsilon}}\bQ_k^\top$
			\ENDFOR
		\end{algorithmic}
	\end{algorithm}
    \vspace*{-0.02cm}\begin{algorithm}[H]
		\caption{Conda \citep{Conda}}
		\label{conda}
		\begin{algorithmic}
			\STATE Hyper parameters: $\eta,\theta,\beta,\varepsilon$
			\STATE Initialize $\X_1$, $\M_0=\0$, $\V_0=\0$.
			\FOR{$k=1,2,\cdots,K$}
			\STATE $\G_k=\mbox{GradOracle}(\X_k)$
			\STATE $\M_k=\theta\M_{k-1}+(1-\theta)\G_k$
            \STATE $\bP_k,\Sigma_k,\mathbf{R}_k^\top=\mbox{SVD}(\M_k)$
			\STATE $\V_k=\beta\V_{k-1}+(1-\beta)(\bP_k^\top\G_k)^2$
			\STATE $\X_{k+1}=\X_k-\eta\bP_k\frac{\bP_k^\top\M_k}{\sqrt{\V_k+\varepsilon}}$
			\ENDFOR
		\end{algorithmic}
	\end{algorithm}
    \vspace*{-0.02cm}\begin{algorithm}[H]
	\caption{Generalized truncated SPlus}
	\label{gen-splus}
	\begin{algorithmic}
		\STATE Hyper parameters: $\eta,\theta$
		\STATE Initialize $\X_1$, $\M_0=\0$.
		\FOR{$k=1,2,\cdots,K$}
		\STATE $\G_k=\mbox{GradOracle}(\X_k)$
		\STATE $\M_k=\theta\M_{k-1}+(1-\theta)\G_k$
        \STATE $\mbox{Generate orthogonal matrices}$
        \STATE $\bP_k\in\mathbb R^{m\times m}\mbox{ and }\bQ_k\in\mathbb R^{n\times n}$
		\STATE $\X_{k+1}=\X_k-\eta\bP_k\frac{\bP_k^\top\M_k\bQ_k}{\max\left\{
\left|\bP_k^\top\M_k\bQ_k\right|,\delta\right\}}\bQ_k^\top$
		\ENDFOR
	\end{algorithmic}
\end{algorithm}
\end{minipage}

\subsection{SOAP}
We present the generalized SOAP algorithm in Algorithm \ref{gsoap}, which includes the original SOAP (Algorithm \ref{soap}) as a special case. In this case, the quantities $\bPh_k$ and $\fO_k$ in Algorithm \ref{general} can be written as
\begin{equation}\label{SOAP-def}
    \bPh_k=\frac{\bP_k^\top\M_k\bQ_k}{\sqrt{\V_k+\varepsilon}},\quad \fO_{k}=\sqrt{\V_k+\varepsilon}.
\end{equation}
The original SOAP periodically constructs orthogonal rotation matrices from the eigenbases of Shampoo’s preconditioners $\bL$ and $\bR$, whereas the generalized form accommodates essentially arbitrary orthogonal matrices, including those used in SOAP as special cases. This enables new classes of rotation matrices that may offer superior practical performance.

\subsection{Conda}\label{conda-section}
Conda is a special case of the generalized SOAP in Algorithm \ref{gsoap}, obtained by choosing $\bP_k\in\mathbb R^{m\times m}$ as the matrix of left singular vectors of $\M_k$ (suppose that $m\leq n$), setting $\bQ_k$ to the identity matrix, and using the same $\bPh_k$ function as in SOAP. The relaxation of the constraints on $\bP_k$ and $\bQ_k$ in generalized SOAP, which allows them to depend on $\G_k$, is especially important here, since Conda constructs $\bP_k$ from the current momentum, which depends on $\G_k$. This illustrates more explicitly why general rotation matrices are needed, enabling Conda to fit naturally within the same framework as SOAP.

\subsection{SPlus}
SPlus adopts the same rotation matrices as SOAP, together with the sign-type update in the rotated space, as presented in Algorithm \ref{splus}. In contrast to SOAP, however, its rotation matrices $\bP_k$ and $\bQ_k$ depend on the current stochastic gradient, whereas SOAP uses one-step-delayed counterparts. We further generalize SPlus to accommodate arbitrary orthogonal matrices in Algorithm \ref{gen-splus}. When the sign-type update is adopted, the function $\bPh_k$ in Algorithm \ref{general} is given by
\begin{equation}\notag
\bPh_k=\operatorname{Sign}\left(\bP_k^\top\M_k\bQ_k\right)=\frac{\bP_k^\top\M_k\bQ_k}{|\bP_k^\top\M_k\bQ_k|},\quad \fO_{k}=|\bP_k^\top\M_k\bQ_k|.
\end{equation}
Though in our analysis, we assume $\fO_{k,i,j}\geq \delta$. This assumption, however, does not always hold for the original sign update in SPlus. SOAP avoids this issue by introducing the stabilization parameter $\varepsilon$. We therefore focus on a truncated modification of the sign function in Algorithm \ref{gen-splus}:
\begin{equation}\label{SPlus-def}
\bPh_k=\frac{\bP_k^\top\M_k\bQ_k}{\fO_k},\qquad \fO_{k,i,j}=\max\left\{
\left|\bP_k^\top\M_k\bQ_k\right|_{i,j},\delta\right\}.
\end{equation}
We consider only the internal iterates and omit the output averaging step used in the original SPlus algorithm. A convergence guarantee for the averaged output would require an additional argument.

\section{Unified Convergence Analysis Framework}\label{converge}
\subsection{Bounding the Normalized Rotated Gradient}
We first establish unified properties for the general framework presented in Algorithm \ref{general}; the argument does not rely on any specific properties of $\fO_k$, except that it is entrywise positive with a lower bound. Following the proofs in \citep{Li-2025-nips}, we obtain the following lemma for Algorithm \ref{general}, which bounds the normalized rotated gradient in expectation.
\begin{lemma}\label{main1}
	Suppose that Assumptions 1-3 and condition $\fO_{k,i,j}\geq\delta$ hold. Let $\hat\sigma^2=\max\left\{\sigma^2,\frac{L\left(f(\X_1)-f^*\right)}{K\gamma^2}\right\}$ with any $\gamma\in(0,1]$, $1-\theta=\sqrt{\frac{L\left(f(\X_1)-f^*\right)}{K\hat\sigma^2}}$,  and $\eta=\sqrt{\frac{\delta^2\left(f(\X_1)-f^*\right)}{4LK\hat\sigma^2}}$.
	Then for the general Algorithm \ref{general}, we have
	\begin{equation}\notag
	\mathcal{S}_K :=\sum_{k=1}^K\E\left[\left\|\frac{\bP_k^\top\nabla f(\X_k)\bQ_k}{\fO_k^{\frac{1}{2}}}\right\|_F^2\right]\leq\frac{7}{\delta}\sqrt{K\hat\sigma^2L\left(f(\X_1)-f^*\right)}.
	\end{equation}
\end{lemma}
In contrast to \citep{Li-2025-nips}, we handle the difficulty posed by orthogonal transformations between the original space and rotated space by exploiting the property $\|\bP^\top\X\bQ\|=\|\X\|$ whenever $\bP$ and $\bQ$ are orthogonal, for any unitarily invariant norm $\|\cdot\|$, as is the case for the nuclear norm and the Frobenius norm. In other words, the Frobenius and nuclear norms of $\X$ are preserved under rotation from the original space to the rotated space.

\subsection{Row-Column Trace-controlling}
If we strictly follow the argument in \citep{Li-2025-nips}, we obtain a convergence rate order of $\bO(\frac{\sqrt{mn}}{K^{1/4}})$ measured by nuclear norm:
$$
\frac{1}{K}\sum_{k=1}^K\E\left[\left\|\nabla f(\X_k)\right\|_*\right]\leq\bO\left(\sqrt{mn}\sqrt[4]{\frac{\hat{\sigma}^2L\left(f(\X_1)-f^*\right)}{K}}\right).
$$
The dimension factor is larger by a factor of $\sqrt{\frac{mn}{m+n}}$ than the convergence rate of Shampoo \citep{shampoo-li-26}. Upon closer inspection, the problematic $\sqrt{mn}$ term stems from the use of an elementwise upper bound. This is not an issue in the analysis of Adam in \citep{Li-2025-nips}, as that work focuses on vector-based optimizers. Matrix-based optimizers, by contrast, require a different analysis that exploits the matrix structure, especially SOAP, which employs an Adam-type update in the rotated space. To address this issue, we propose a row-column trace-control argument that separately controls the row-wise and column-wise maxima. Namely, denote
\begin{eqnarray}
\begin{aligned}\label{define-B-C}
	b_{k, i}\geq\max _{1 \leq j \leq n}\fO_{k,i,j},& \quad \mathbf{B}_{k}= \operatorname{diag}\left(b_{k, 1}, b_{k, 2}, \ldots b_{k, m}\right),  \\
	c_{k, j}\geq\max _{1 \leq i \leq m}\fO_{k,i,j},& \quad  \mathbf{C}_{k}= \operatorname{diag}\left(c_{k, 1}, c_{k, 2} \cdots c_{k, n}\right). 
\end{aligned}
\end{eqnarray}
With the control matrices defined above, we obtain the following lemma. The idea is that introducing control matrices allows the desired convergence rate for the nuclear norm of the gradient to decompose into (i) a term $\mathcal{S}_k$, already bounded in Lemma \ref{main1}, and (ii) the traces of the two control matrices, $\E\left[\operatorname{tr} \mathbf{B}_{k}\right]$ and $\E\left[\operatorname{tr} \mathbf{C}_{k}\right]$. Thus, we reduce the problem to bounding these traces for different optimizers. 
\begin{lemma}\label{lemma-nuclear}
For the general Algorithm \ref{general} and $\mathbf{B}_{k}$ and $\mathbf{C}_{k}$ defined in (\ref{define-B-C}), we have
\begin{equation}\label{nuclear}
\left(\sum_{k=1}^{K} \E \left[ \left\| \nabla f(\X_{k})\right\|_{*}\right]\right)^{2}\hspace*{-0.08cm}\leq\hspace*{-0.05cm} \frac{1}{2}\hspace*{-0.05cm} \left(\sum_{k=1}^{K}\E\left[\operatorname{tr}\mathbf{B}_{k}\right]\hspace*{-0.05cm}+\hspace*{-0.05cm}\sum_{k=1}^{K}\E\left[\operatorname{tr}\mathbf{C}_{k}\right]\right)\hspace*{-0.05cm}\cdot\hspace*{-0.05cm} \sum_{k=1}^K\E\hspace*{-0.05cm}\left[\left\|\frac{\bP_k^\top\nabla f(\X_k)\bQ_k}{\fO_k^{\frac{1}{2}}}\right\|_F^2\right]\hspace*{-0.05cm}.
\end{equation}
\end{lemma}
We explain the rationale for choosing the control matrices to be of the row-wise and column-wise maxima diagonal form. As proved in Appendix \ref{appendix-a}, by employing unitary invariance of nuclear norm, Schatten-Hölder inequality, and the  Hölder's inequality, we obtain
\begin{eqnarray}
\begin{aligned}\notag
	\left(\sum_{k=1}^{K} \E \left[ \left\| \nabla f(\X_{k})\right\|_{*}\right]\right)^{2}\hspace*{-0.11cm}\leq\hspace*{-0.09cm}
	\sqrt{\sum_{k=1}^{K} \E\left[\operatorname{tr} \mathbf{B}_{k}\right] \cdot \sum_{k=1}^{K} \E\left[\operatorname{tr} \mathbf{C}_{k}\right]} \cdot \sum_{k=1}^{K} \E\left[\left\|\mathbf{B}_{k}^{-\frac{1}{4}} \bP_k^\top \nabla f(\X_{k})\bQ_k \mathbf{C}_{k}^{-\frac{1}{4}}\right\|_F^{2}\right]\hspace*{-0.05cm}. 
\end{aligned}
\end{eqnarray}
Then, for the last term, we have
\begin{eqnarray}
\begin{aligned}\notag
    \left\|\mathbf{B}_{k}^{-\frac{1}{4}} \bP_k^\top \nabla f(\X_{k})\bQ_k \mathbf{C}_{k}^{-\frac{1}{4}}\right\|_F^{2}=&\sum_{i=1}^{m} \sum_{j=1}^{n}\left(\frac{\left(\bP_k^\top \nabla f(\X_{k})\bQ_k\right)_{i,j}}{\sqrt[4]{b_{k,i}c_{k,j}}}\right)^{2}\\
    \leq&\sum_{i=1}^{m} \sum_{j=1}^{n}\left(\frac{\bP_k^\top \nabla f(\X_{k})\bQ_k}{\fO_k^{\frac{1}{2}}}\right)_{i,j}^{2}=\left\|\frac{\bP_k^\top\nabla f(\X_k)\bQ_k}{\fO_k^{\frac{1}{2}}}\right\|_F^2,
\end{aligned}
\end{eqnarray}
where we use the diagonal forms of $\mathbf{B}$ and $\mathbf{C}$, together with $b_{k, i}\geq\max _{1 \leq j \leq n}\fO_{k,i,j}$ and $c_{k, j}\geq\max _{1 \leq i \leq m}\fO_{k,i,j}$. We finally convert $\nabla f(\X_{k})$ into $\frac{\bP_k^\top\nabla f(\X_k)\bQ_k}{\fO_k^{\frac{1}{2}}}$.

The reason adding control matrices yields a sharper bound lies in the number of terms that must be controlled. In other words, the bound improves because row/column trace aggregation replaces $mn$ coordinate-wise controls with $m+n$ trace controls, providing more refined control than elementwise bounding. From this point on, we examine several algorithms through the lens of this framework.

\subsection{Algorithms with EMA-based Normalizers}
We say that an algorithm has EMA-based normalizers if its associated matrix $\fO_k$ necessarily incorporates EMA (exponential moving average) values in addition to the input $\bP_k^\top\M_k\bQ_k$. 

\subsubsection{SOAP}
SOAP is a canonical example of an EMA-based algorithm: Bt \eqref{SOAP-def}, 
$\fO_{k}=\sqrt{\V_k+\varepsilon}$, which involves the EMA term $\V_k$. By exploiting the recursive update of $\V_k$, the following lemma bounds the traces of the control matrices in SOAP.
\begin{lemma}\label{trace-control-soap}
	Suppose that Assumption 3 holds. Let $\hat\sigma^2=\max\left\{\sigma^2,\frac{L\left(f(\X_1)-f^*\right)}{K\gamma^2}\right\}$ with any $\gamma\in(0,1]$, $0<\beta< 1$ and $\V_0=\0$. 
	Then for generalized SOAP (Algorithm \ref{gsoap}), we have
    \begin{eqnarray}
	\begin{aligned}\label{soap-lemma}
		\sum_{k=1}^{K}\E\left[\operatorname{tr} \mathbf{B}_{k}\right]&\leq K\left(\sqrt{2m}\hat{\sigma}+m\sqrt{\varepsilon}\right)+4\sum_{k=1}^{K} \E\left[\left\|\frac{\bP_k^\top \nabla f(\X_{k})\bQ_k}{\sqrt[4]{\V_{k}+\varepsilon}}\right\|_{F}^{2}\right],\\
		\sum_{k=1}^{K}\E\left[\operatorname{tr} \mathbf{C}_{k}\right]& \leq K\left(\sqrt{2n}\hat{\sigma}+n\sqrt{\varepsilon}\right)+4\sum_{k=1}^{K} \E\left[\left\|\frac{\bP_k^\top \nabla f(\X_{k})\bQ_k}{\sqrt[4]{\V_{k}+\varepsilon}}\right\|_{F}^{2}\right].
	\end{aligned}
    \end{eqnarray}
\end{lemma}
\textbf{Design of the Control Matrices}: We focus on the row-controlling and the column construction is analogous. First, we begin designing the control matrix as follows:
\begin{gather*}
		\phi_{k, i}=\max _{1 \leq j \leq n}\left|\bP_k^\top \nabla f(\X_{k})\bQ_k\right|_{i, j}, \quad \varphi_{k,i}=\max _{1 \leq j \leq n}\left|\bP_k^\top \left(\G_{k}-\nabla f\left(\X_{k}\right)\right)\bQ_k\right|_{i, j},\\
		\bar{\alpha}_{k, i}=\beta \bar{\alpha}_{k-1, i}+2(1-\beta)\left(\phi_{k, i}^2+\varphi_{k,i}^2\right)\mbox{ with }\bar{\alpha}_{0, i}=0,\quad b_{k,i}=\sqrt{\bar{\alpha}_{k,i}+\varepsilon}.
\end{gather*}
Then from the recursion of $\V_k$ in Algorithm \ref{gsoap} and by induction, we have
\begin{gather*}
        \bar{\alpha}_{k,i}\geq \V_{k,i,j}\quad\mbox{and}\quad b_{k,i}\geq\sqrt{\V_{k,i,j}+\varepsilon}=\fO_{k,i,j}\mbox{ for all }j=1,2,\cdots,n.
\end{gather*}
Thus, the coefficients $b_{k,i}$ satisfy (\ref{define-B-C}). Thus it is a valid choice for the trace-control process. Next, we upper bound it. To exclude the noise variance of $b_k$, we further define
\begin{gather*}
        \bar{\varphi}_{k, i}=\beta \bar{\varphi}_{k-1, i}+2(1-\beta)\varphi_{k,i}^2\mbox{ with }\bar{\varphi}_{0, i}=0,\quad n_{k,i}=\sqrt{\bar{\varphi}_{k,i}+\varepsilon}.
\end{gather*}
Then, it is possible to construct a recursion inequality
$$
b_{k,i}-n_{k,i}\leq\sqrt{\beta}(b_{k-1,i}-n_{k-1,i})+2(1-\beta)\frac{\phi_{k,i}^2}{b_{k,i}}.
$$
This further gives an upper bound of $b_k$:
$$
	b_{k,i}\leq n_{k,i}+2(1-\beta) \sum_{t=1}^k \beta^{\frac{k-t}{2}}\frac{\phi_{t,i}^2}{b_{t,i}}.
	$$
Summing over $k$ and $i$ and taking expectation, the $\sum_{k=1}^K\sum_{i=1}^m \E[n_{k,i}]$ term converts into $\sqrt{2Km\sum_{k=1}^K\sum_{i=1}^m \E[\varphi_{k,i}^2]}+Km\sqrt{\varepsilon}$, while the latter $2(1-\beta) \sum_{k=1}^K\sum_{i=1}^m\sum_{t=1}^k \beta^{\frac{k-t}{2}}\E\left[\frac{\phi_{t,i}^2}{b_{t,i}}\right]$ turns out to be $4\sum_{k=1}^K\sum_{i=1}^m\E\left[\frac{\phi_{k,i}^2}{b_{k,i}}\right]$. By employing
\begin{align}
    \sum_{k=1}^K\sum_{i=1}^m\E\left[\varphi_{k,i}^2\right]&\leq \sum_{k=1}^K\sum_{i=1}^m\sum_{j=1}^n\E\left[\left(\bP_k^\top\left(\G_k-\nabla f(\X_{k})\right)\bQ_k\right)_{i,j}^2\right]\notag\\
    &=\sum_{k=1}^K\E\left[\left\|\bP_k^\top\left(\G_k-\nabla f(\X_{k})\right)\bQ_k\right\|_F^2\right]=\sum_{k=1}^K\E\left[\left\|\G_k-\nabla f(\X_{k})\right\|_F^2\right]\leq K\sigma^2\notag
\end{align}
and
\begin{align}
\sum_{k=1}^K\sum_{i=1}^m\frac{\phi_{k,i}^2}{b_{k,i}}\leq \sum_{k=1}^K\sum_{i=1}^m\sum_{j=1}^n\frac{\left(\bP_k^\top\nabla f(\X_{k})\bQ_k\right)_{i,j}^2}{b_{k,i}} \leq\sum_{k=1}^K\sum_{i=1}^m\sum_{j=1}^n\left(\frac{\bP_k^\top\nabla f(\X_{k})\bQ_k}{\sqrt[4]{\V_k+\varepsilon}}\right)_{i,j}^2,\notag
\end{align}
we have the conclusion. \hfill $\blacksquare$

With the above supporting lemmas, we get to the following theorem establishing the convergence of SOAP.
\begin{theorem}\label{SOAP}
Suppose that Assumptions 1-3 and condition $\sqrt{\V_{k, i, j}+\varepsilon}\geq\delta$
hold. Let $\hat\sigma^2=\max\left\{\sigma^2,\frac{L\left(f(\X_1)-f^*\right)}{K\gamma^2}\right\}$ with any $\gamma\in(0,1]$, $1-\theta=\sqrt{\frac{L\left(f(\X_1)-f^*\right)}{K\hat\sigma^2}}$, $0<\beta< 1$, $\eta=\sqrt{\frac{\delta^2\left(f(\X_1)-f^*\right)}{4LK\hat\sigma^2}}$, and $\varepsilon=\frac{\tau\hat{\sigma}^2}{m+n}$, where $\tau$ is a tuning factor with $\tau \in (0,1]$. Then for generalized SOAP (Algorithm \ref{gsoap}), we have
\begin{eqnarray}
\begin{aligned}\notag
\frac{1}{K}\sum_{k=1}^K\E\left[\left\|\nabla f(\X_k)\right\|_*\right] \leq \left(2\sqrt{m+n}+16\frac{\hat{\sigma}}{\delta}\right)\cdot \sqrt[4]{\frac{ \hat{\sigma}^{2} L\left(f\left(\mathbf{X}_{1}\right)-f^{*}\right)}{K}}.
\end{aligned}
\end{eqnarray} 
\end{theorem}

When $\delta=\sqrt{\varepsilon}=\sqrt{\frac{\tau}{m+n}}\hat\sigma$, the right-hand side of the above convergence rate becomes 
\begin{align}
&\left(2+\frac{16}{\sqrt{\tau}}\right)\sqrt{m+n}\cdot \sqrt[4]{\frac{ \hat{\sigma}^{2} L\left(f\left(\mathbf{X}_{1}\right)-f^{*}\right)}{K}}\notag\\
&\qquad=\left(2+\frac{16}{\sqrt{\tau}}\right)\sqrt{m+n}\cdot \max\left\{\sqrt[4]{\frac{ \sigma^{2} L\left(f\left(\mathbf{X}_{1}\right)-f^{*}\right)}{K}},\sqrt{\frac{ L\left(f\left(\mathbf{X}_{1}\right)-f^{*}\right)}{K\gamma}}\right\},\notag
\end{align}
which is exactly the same order as Shampoo’s convergence rate \citep{shampoo-li-26}. If we further set $\tau=\Theta(1)$, then, since $\|\nabla f(\X_k)\|_F\leq\|\nabla f(\X_k)\|_*\leq\sqrt{\min\{m,n\}}\|\nabla f(\X_k)\|_F$ , this rate is comparable to the optimal convergence rate of SGD and matches the lower bound of stochastic nonconvex optimization \citep{Arjevani-2023-mp} under the ideal condition $\|\nabla f(\X_k)\|_*=\Theta(\sqrt{\min\{m,n\}})\|\nabla f(\X_k)\|_F$ and balanced $m$ and $n$. 

Nevertheless, even for modern industrial large language models, the dimensions $m$ and $n$ are not particularly large. For example, in GPT-3 with $175$ billion parameters, the QKV projection matrices have dimensions $m=n=12288$, while the weight matrices in the feed-forward layer have dimensions $(m,n)=(12288,49152)$ or $(49152,12288)$. So $\frac{\hat{\sigma}^2}{m+n}$ may not be a very small constant. Consequently, our parameter $\varepsilon=\frac{\tau\hat{\sigma}^2}{m+n}$ with $\tau=\Theta(1)$ is set larger than what is typically used in practice. However, choosing $\tau$ small would slow the theoretical convergence rate.

On the other hand, whenever $\V_{k, i, j}=\Theta(\frac{\hat{\sigma}^2}{m+n})$, we have $\delta^2=\Theta(\frac{\hat{\sigma}^2}{m+n})$ regardless of the value chosen for $\varepsilon$, and the convergence rate remains unaffected. This permits $\varepsilon$ to be chosen much smaller than $\delta^2$, consistent with practical settings.

\subsubsection{Adam}

Theorem \ref{SOAP} holds for arbitrary orthogonal matrices $\bP_k$ and $\bQ_k$. In particular, setting $\bP_k=\I_m$ and $\bQ_k=\I_n$ reduces SOAP to Adam when applied to problems whose parameters are matrices, so Theorem \ref{SOAP} also applies to Adam. This observation requires no separate convergence argument, as it follows directly from the general framework by choosing the identity matrices as the orthogonal bases. Thus, we establish a nuclear-norm convergence guarantee for Adam for the first time, whereas existing analyses of Adam under comparable assumptions typically report bounds in vector $\ell_1$ norm \citep{Li-2025-nips} or $\ell_2$ norm \citep{haochuanli-2023}. For example, when applied to matrix-valued parameters, the analysis in \citep{Li-2025-nips} establishes the following convergence rate for Adam
\begin{eqnarray}
\begin{aligned}\notag
	\frac{1}{K}\sum_{k=1}^K\E\left[\left\|\nabla f(\X_k)\right\|_1\right]\leq \bO\left(\sqrt{m n}\sqrt[4]{\frac{\hat{\sigma}^2L\left(f(\X_1)-f^*\right)}{K}}\right).
\end{aligned}
\end{eqnarray}
In contrast, our new convergence rate is measured by the nuclear norm on the left-hand side and reduces $\sqrt{mn}$ to $\sqrt{m+n}$ on the right-hand side, at the cost of requiring a larger value of $\varepsilon$; \citet{Li-2025-nips} sets only $\varepsilon=\frac{\hat{\sigma}^2}{mn}$. This leads to a striking finding: elementwise Adam attains a nuclear-norm convergence rate identical to that of AdamW-style Shampoo \citep{shampoo-li-26}.

\subsubsection{Conda}
Conda provides another useful illustration of our unified analysis framework: Theorem \ref{SOAP} applies directly to Conda with the choice of rotation matrices specified in Section \ref{conda-section}. More importantly, it shows that the row-column trace-control argument has a broad range of applicability, spanning simple choices such as that in Conda as well as more complicated ones such as in SOAP.

\subsection{Algorithms with Instantaneous normalizers}
In contrast to the types of algorithms discussed above, there exist algorithms that do not require EMA in $\fO_k$. We refer to these as algorithms with instantaneous normalizers. They have the favorable property that less data needs to be stored and processed, which helps alleviate the memory burden. 
\subsubsection{SPlus}
The formal definition of $\bPh_k$ and $\fO_k$ studied in this paper is shown in \eqref{SPlus-def}. Since no EMA values are involved, directly following the proofs of SOAP to form a recursion becomes difficult, which forces us to take a different route (see Appendix \ref{appendix-c} for more details). Nevertheless, we establish the convergence rate of the truncated SPlus in the following theorem.
\begin{theorem}\label{SPlusCon}
	Suppose that Assumptions 1-3
	hold. Let $\hat\sigma^2=\max\left\{\sigma^2,\frac{L\left(f(\X_1)-f^*\right)}{K\gamma^2}\right\}$ with any $\gamma\in(0,1]$, $1-\theta=\sqrt{\frac{L\left(f(\X_1)-f^*\right)}{K\hat\sigma^2}}$, $\eta=\sqrt{\frac{\delta^2\left(f(\X_1)-f^*\right)}{4LK\hat\sigma^2}}$, and $\delta=\frac{\tau\hat{\sigma}}{\sqrt{m+n}}$, where $\tau$ is a tuning factor with $\tau \in (0,1]$. Then for generalized truncated SPlus  (Algorithm \ref{gen-splus}), we have
	\begin{eqnarray}
	\begin{aligned}\notag
		\frac{1}{K}\sum_{k=1}^K\E\left[\left\|\nabla f(\X_k)\right\|_*\right] \leq \left(4+\frac{11}{\tau}\right)\sqrt{m+n}\cdot \sqrt[4]{\frac{ \hat{\sigma}^{2} L\left(f\left(\mathbf{X}_{1}\right)-f^{*}\right)}{K}}.
	\end{aligned}
	\end{eqnarray} 
\end{theorem}

\textbf{Design of the Control Matrices}: We only focus on the row-controlling. Define
$$
j_i^*=\argmax_{1\leq j \leq n}\left|\bP_k^\top\M_k\bQ_k\right|_{i,j},\quad b_{k,i}=\left|\bP_k^\top\M_k\bQ_k\right|_{i,j_i^*}+\delta\geq\max_{1 \leq j \leq n}\fO_{k,i,j}
$$
for (\ref{SPlus-def}). Therefore,
\begin{eqnarray}
\begin{aligned}\notag
b_{k,i}\leq& \left|\bP_k^\top\nabla f(\X_k)\bQ_k\right|_{i,j_i^*}+\left|\bP_k^\top\left(\M_k-\nabla f(\X_k) \right)\bQ_k\right|_{i,j_i^*}+\delta\\
\leq& \frac{\left|\bP_k^\top\nabla f(\X_k)\bQ_k\right|_{i,j_i^*}^2}{2b_{k,i}}+\frac{b_{k,i}}{2}+\left|\bP_k^\top\left(\M_k-\nabla f(\X_k) \right)\bQ_k\right|_{i,j_i^*}+\delta\\
\leq&\sum_{j=1}^n\frac{\left|\bP_k^\top\nabla f(\X_k)\bQ_k\right|_{i,j}^2}{2\fO_{k,i,j}}+\frac{b_{k,i}}{2}+\left\|\left(\bP_k^\top\left(\M_k-\nabla f(\X_k) \right)\bQ_k\right)_{i,:}\right\|_2+\delta,
\end{aligned}
\end{eqnarray}
which gives an upper bound of $b_{k,i}$. Similar to the discussion below Lemma \ref{trace-control-soap}, using $$
	\sum_{k=1}^K\E\left[\left\|\M_k-\nabla f(\X_k) \right\|_F\right]\leq \sqrt{13}K\hat{\sigma}
	$$ established in Lemma \ref{error-bounding}, we can finally obtain an upper bound on $\sum_{k=1}^{K}\E\left[\operatorname{tr} \mathbf{B}_{k}\right]$ as 
    \begin{eqnarray}
    \begin{aligned}\notag
		\sum_{k=1}^{K}\E\left[\operatorname{tr} \mathbf{B}_{k}\right]&\leq 2\sqrt{13m}K\hat{\sigma}+2Km\delta+\sum_{k=1}^K\E\left[\left\|\frac{\bP_k^\top\nabla f(\X_k)\bQ_k}{\fO_k^{\frac{1}{2}}}\right\|_F^2\right].
	\end{aligned}
    \end{eqnarray}
    \hfill $\blacksquare$

\section{Conclusion}

This paper provides a unified convergence analysis framework for rotation-based matrix optimizers, including SOAP, Conda, and truncated SPlus. The framework permits arbitrary orthogonal rotations that may depend on the current stochastic gradient andyields convergence bounds in terms of the nuclear norm of the gradient for generalized SOAP, Conda, truncated SPlus, and matrix-parameterized Adam. Technically, our main approach leverages a row-column trace-control argument that decomposes the control term into row-wise and column-wise maxima and converts elementwise bounds into bounds on two diagonal control matrices. The framework can be applied to other rotation-based updates whenever their normalizers admit a uniform positive lower bound and suitable row and column trace bounds.

\textbf{Limitations.} Our theorems rely on parameter choices that depend on problem-dependent quantities, such as the initial optimality gap and the noise bound. Our analysis requires a uniform positive lower bound on $\fO_k$. Although SOAP and the truncated SPlus satisfy this condition, it imposes an additional restriction on other rotation-based optimizers, such as the untruncated SPlus. The SPlus result pertains to its internal iterates; extending the analysis to averaged outputs remains open.

\bibliography{iclr2027_conference}
\bibliographystyle{icml2020}

\newpage

\appendix

\section{Proofs for the General Framework}\label{appendix-a}
We first give the H\"older's inequality and Schatten-H\"older inequality, as well as two standard identities for Schatten norms.

\begin{lemma}[H\"older's Inequality]
    For all $x_1,x_2,\cdots,x_n,y_1,y_2,\cdots,y_n\geq0$, if $\frac{1}{p}+\frac{1}{q}=1$, we have
    $$
        \sum_{i=1}^n x_iy_i \leq \left(\sum_{i=1}^n x_i^p\right)^{\frac{1}{p}}\cdot\left(\sum_{i=1}^n y_i^q\right)^{\frac{1}{q}}.
    $$
\end{lemma}
\begin{lemma}[Schatten-H\"older inequality]\citep{{rajendra-book}}
    For $\X\in \R^{m\times n}$, $\Y\in \R^{n\times s}$, and $\Z\in \R^{s\times t}$, if $\frac{1}{p}+\frac{1}{q}+\frac{1}{r}=1$, we have
    $$
        \left\|\X\Y\Z\right\|_{S_1}\leq \left\|\X\right\|_{S_p}\left\|\Y\right\|_{S_q}\left\|\Z\right\|_{S_r},
    $$
    where the Schatten-$p$ norm is defined as
    $$
          \|\X\|_{S_p}=\left\{\begin{array}{cc}
               \Big(\sum_{i=1}^{\min\{m,n\}}\sigma_i(\X)^p\Big)^{1/p}, & 1\leq p<\infty,\\
          \sigma_1(\X), & p=\infty.
           \end{array}\right.
    $$
\end{lemma}

\begin{lemma}\label{trace-transform}
    For any diagonal matrix $\X\in\mathbb R^{n\times n}$ with positive entries, $\Y\in\mathbb R^{m\times n}$, and $p>1$, we have
    $$\left\|\X^{\frac{1}{p}}\right\|_{S_{p}}^{p}=\operatorname{tr} \X,\qquad \|\Y\|_{S_{2}}^{2}=\|\Y\|_F^2.$$
\end{lemma}
\begin{proof}
    Recall that we denote $\sigma_i$ and $\lambda_i$ as the $i$-th singular value and eigenvalue, respectively. Therefore $\sigma_i\left(\X^{\frac{1}{p}}\right)=\sigma_i\left(\X\right)^{\frac{1}{p}}$ and $\sigma_i(\X)=\lambda_i(\X)$. So we have
    $$
        \left\|\X^{\frac{1}{p}}\right\|_{S_{p}}^{p}=\left(\sum_{i=1}^n \lambda_i\left(\X\right)^{\frac{1}{p}\cdot p}\right)^{\frac{1}{p}\cdot p}=\sum_{i=1}^n \lambda_i\left(\X\right)=\operatorname{tr} \X,
    $$
    and
    $$
        \left\|\Y\right\|_{S_{2}}^{2}=\left(\sum_{i=1}^{\min\{m,n\}} \sigma_i\left(\Y\right)^2\right)^{\frac{1}{2}\cdot 2}=\operatorname{tr} \Y\Y^\top=\|\Y\|_F^2.
    $$
    \hfill $\blacksquare$
\end{proof}
\begin{remark}
Throughout this paper, we define $\X^{\frac{1}{p}}$ as the element‑wise $\frac{1}{p}$-th power. To avoid having to define the matrix power, we prove the first equality only for diagonal $\X$, although it holds for any symmetric positive semi-definite matrix.
\end{remark}

\begin{definition}[Unitarily invariant norm]\citep{{rajendra-book}}
A norm is unitarily invariant if it satisfies $\|\bP\A\bQ^\top\|=\|\A\|$ for orthogonal matrices $\bP$ and $\bQ$. Both the Frobenius norm and nuclear norm are unitarily invariant norms.
\end{definition}

The proof of Lemma \ref{main1} presented below follows the technical framework of AdamW established in \citep{Li-2025-nips}. The difference lies in that we extend the argument to accommodate orthogonal rotations, which is the main difficulty in analyzing rotation-based optimizers such as SOAP.
\begin{proof}[Proof of Lemma \ref{main1}]
	As the gradient is $L$-Lipschitz, we have
	\begin{eqnarray}
	\hspace*{-3cm}\begin{aligned}\notag
		&f(\X_{k+1})-f(\X_k)\\
		\leq&\<\nabla f(\X_k),\X_{k+1}-\X_k\>+\frac{L}{2}\|\X_{k+1}-\X_k\|_F^2\\
		=&-\eta\<\nabla f(\X_k),\bP_k\frac{\bP_k^\top\M_k\bQ_k}{\fO_k}\bQ_k^\top\>+\frac{L\eta^2}{2}\left\|\bP_k\frac{\bP_k^\top\M_k\bQ_k}{\fO_k}\bQ_k^\top\right\|_F^2\\
		=&-\eta\<\bP_k^\top\nabla f(\X_k)\bQ_k,\frac{\bP_k^\top\M_k\bQ_k}{\fO_k}\>+\frac{L\eta^2}{2}\left\|\bP_k\frac{\bP_k^\top\M_k\bQ_k}{\fO_k}\bQ_k^\top\right\|_F^2
         \end{aligned}
	\end{eqnarray}
    \begin{eqnarray}
	\begin{aligned}\label{equ1}
        \overset{a}=&-\eta\<\frac{\bP_k^\top\nabla f(\X_k)\bQ_k}{\fO_k^{\frac{1}{2}}},\frac{\bP_k^\top\M_k\bQ_k}{\fO_k^{\frac{1}{2}}}\>+\frac{L\eta^2}{2}\left\|\frac{\bP_k^\top\M_k\bQ_k}{\fO_k}\right\|_F^2\\
		\overset{b}\leq&-\eta\<\frac{\bP_k^\top\nabla f(\X_k)\bQ_k}{\fO_k^{\frac{1}{2}}},\frac{\bP_k^\top\M_k\bQ_k}{\fO_k^{\frac{1}{2}}}\>+\frac{L\eta^2}{2\delta}\left\|\frac{\bP_k^\top\M_k\bQ_k}{\fO_k^{\frac{1}{2}}}\right\|_F^2\\
		=&-\frac{\eta}{2}\left\|\frac{\bP_k^\top\nabla f(\X_k)\bQ_k}{\fO_k^{\frac{1}{2}}}\right\|_F^2-\frac{\eta}{2}\left\|\frac{\bP_k^\top\M_k\bQ_k}{\fO_k^{\frac{1}{2}}}\right\|_F^2\\
		&+\frac{\eta}{2}\left\|\frac{\bP_k^\top\nabla f(\X_k)\bQ_k}{\fO_k^{\frac{1}{2}}}-\frac{\bP_k^\top\M_k\bQ_k}{\fO_k^{\frac{1}{2}}}\right\|_F^2+\frac{L\eta^2}{2\delta}\left\|\frac{\bP_k^\top\M_k\bQ_k}{\fO_k^{\frac{1}{2}}}\right\|_F^2\\
		\overset{c}\leq&-\frac{\eta}{2}\left\|\frac{\bP_k^\top\nabla f(\X_k)\bQ_k}{\fO_k^{\frac{1}{2}}}\right\|_F^2-\frac{\eta}{4}\left\|\frac{\bP_k^\top\M_k\bQ_k}{\fO_k^{\frac{1}{2}}}\right\|_F^2+\frac{\eta}{2\delta}\left\|\bP_k^\top\nabla f(\X_k)\bQ_k-\bP_k^\top\M_k\bQ_k\right\|_F^2\\
		\overset{d}=&-\frac{\eta}{2}\left\|\frac{\bP_k^\top\nabla f(\X_k)\bQ_k}{\fO_k^{\frac{1}{2}}}\right\|_F^2-\frac{\eta}{4}\left\|\frac{\bP_k^\top\M_k\bQ_k}{\fO_k^{\frac{1}{2}}}\right\|_F^2+\frac{\eta}{2\delta}\left\|\nabla f(\X_k)-\M_k\right\|_F^2,
	\end{aligned}
	\end{eqnarray}
	where we use the notation that $\frac{\X}{\Y}$ and $\X^{\frac{1}{2}}$ are the elementwise division and square-root in $\overset{a}=$, $\overset{b}\leq$, and $\overset{c}\leq$, the fact that $\|\cdot\|_F$ is a unitarily invariant norm such that $\|\bP\A\bQ^\top\|_F=\|\A\|_F$ for orthogonal matrices $\bP$ and $\bQ$ in $\overset{a}=$ and $\overset{d}=$, the condition $\fO_{k,i,j}\geq \delta$ in $\overset{b}\leq$ and $\overset{c}\leq$, and let $\eta\leq\frac{\delta}{2L}$ in $\overset{c}\leq$. Taking expectation on both sides of (\ref{equ1}) with respect to $\zeta_k$ conditioned on $\F_{k-1}$, multiplying both sides of (\ref{GM-dif-equ}) by $\frac{\eta}{2\delta(1-\theta)}$, and adding to (\ref{equ1}), we have
	\begin{eqnarray}
	\begin{aligned}\label{equ3}
		&\E_k\left[f(\X_{k+1})-f^*+\frac{\eta\theta}{2\delta(1-\theta)}\left\|\nabla f(\X_k)-\M_k\right\|_F^2+\frac{\eta}{4}\left\|\frac{\bP_k^\top\M_k\bQ_k}{\fO_k^{\frac{1}{2}}}\right\|_F^2\big| \F_{k-1}\right]\\
		\leq& f(\X_k)-f^*-\frac{\eta}{2}\E_k\left[\left\|\frac{\bP_k^\top\nabla f(\X_k)\bQ_k}{\fO_k^{\frac{1}{2}}}\right\|_F^2\big| \F_{k-1}\right]\\
		&+\frac{\eta\theta}{2\delta(1-\theta)}\left\|\M_{k-1}-\nabla f(\X_{k-1})\right\|_F^2+\frac{L^2\eta^3}{2\delta^2(1-\theta)^2}\left\|\frac{\bP_{k-1}^\top\M_{k-1}\bQ_{k-1}}{\fO_{k-1}^{\frac{1}{2}}}\right\|_F^2+\frac{\eta(1-\theta)\sigma^2}{2\delta}\\
		\leq& f(\X_k)-f^*-\frac{\eta}{2}\E_k\left[\left\|\frac{\bP_k^\top\nabla f(\X_k)\bQ_k}{\fO_k^{\frac{1}{2}}}\right\|_F^2\big| \F_{k-1}\right]\\
		&+\frac{\eta\theta}{2\delta(1-\theta)}\left\|\M_{k-1}-\nabla f(\X_{k-1})\right\|_F^2+\frac{\eta}{4}\left\|\frac{\bP_{k-1}^\top\M_{k-1}\bQ_{k-1}}{\fO_{k-1}^{\frac{1}{2}}}\right\|_F^2+\frac{\eta(1-\theta)\sigma^2}{2\delta},
	\end{aligned}
	\end{eqnarray}
	where we let $\eta^2\leq\frac{\delta^2(1-\theta)^2}{2L^2} $ in the last inequality. Taking expectation with respect to $\F_{k-1}$ and summing (\ref{equ1}) with $k=1$ and (\ref{equ3}) over $k=2,3,\cdots,K$, we have
	\begin{eqnarray}
	\begin{aligned}\notag
		&\E_{\F_K}\left[f(\X_{K+1})-f^*+\frac{\eta\theta}{2\delta(1-\theta)}\left\|\nabla f(\X_K)-\M_K\right\|_F^2+\frac{\eta}{4}\left\|\frac{\bP_K^\top\M_K\bQ_K}{\fO_K^{\frac{1}{2}}}\right\|_F^2\right]\\
		\leq& f(\X_1)-f^*-\frac{\eta}{2}\sum_{k=1}^K\E_{\F_{k}}\left[\left\|\frac{\bP_k^\top\nabla f(\X_k)\bQ_k}{\fO_k^{\frac{1}{2}}}\right\|_F^2\right]\\
		&+\left(\frac{\eta}{2\delta}+\frac{\eta\theta}{2\delta(1-\theta)}\right)\E_{\F_1}\left[\left\|\M_1-\nabla f(\X_1)\right\|_F^2\right]+ \frac{(K-1)\eta(1-\theta)\sigma^2}{2\delta}
    \end{aligned}
	\end{eqnarray}
    \begin{eqnarray}
	\hspace*{-2.5cm}\begin{aligned}\label{equ6}
		=&f(\X_1)-f^*-\frac{\eta}{2}\sum_{k=1}^K\E_{\F_{k}}\left[\left\|\frac{\bP_k^\top\nabla f(\X_k)\bQ_k}{\fO_k^{\frac{1}{2}}}\right\|_F^2\right]\\
		&+\frac{\eta}{2\delta(1-\theta)}\E_{\F_1}\left[\left\|\M_1-\nabla f(\X_1)\right\|_F^2\right]+\frac{(K-1)\eta(1-\theta)\sigma^2}{2\delta}.
	\end{aligned}
	\end{eqnarray}
	As the gradient is $L$-Lipschitz, we have
	\begin{eqnarray}
	\begin{aligned}\notag
		f^*&\leq f\left(\X-\frac{1}{L}\nabla f(\X)\right)\leq f(\X)-\frac{1}{L}\<\nabla f(\X),\nabla f(\X)\>+\frac{L}{2}\left\|\frac{1}{L}\nabla f(\X)\right\|_F^2\\
		&=f(\X)-\frac{1}{2L}\left\|\nabla f(\X)\right\|_F^2.
	\end{aligned}
	\end{eqnarray}
	Using the recursion of $\M_1$ and $\M_0=\0$, as well as Assumptions 2-3, we have
	\begin{eqnarray}
	\begin{aligned}\label{variancebound}
		\E_{\F_1}\left[\left\|\nabla f(\X_1)-\M_1\right\|_F^2\right]=&\E_{\F_1}\left[\left\|\theta\nabla f(\X_1)+(1-\theta)\left(\nabla f(\X_1)-\G_1\right)\right\|_F^2\right]\\
		=&\theta^2\left\|\nabla f(\X_1)\right\|_F^2+(1-\theta)^2\E_{\F_1}\left[\left\|\nabla f(\X_1)-\G_1\right\|_F^2\right]\\
		\leq& 2L\left(f(\X_1)-f^*\right) + (1-\theta)^2\sigma^2.
	\end{aligned}
	\end{eqnarray}
    Plugging into (\ref{equ6}), we have
	\begin{eqnarray}
	\begin{aligned}\notag
		&\E_{\F_K}\left[f(\X_{K+1})-f^*+\frac{\eta\theta}{2\delta(1-\theta)}\left\|\nabla f(\X_K)-\M_K\right\|_F^2+\frac{\eta}{4}\left\|\frac{\bP_K^\top\M_K\bQ_K}{\fO_K^{\frac{1}{2}}}\right\|_F^2\right]\\
		\leq&f(\X_1)-f^*-\frac{\eta}{2}\sum_{k=1}^K\E_{\F_{k}}\left[\left\|\frac{\bP_k^\top\nabla f(\X_k)\bQ_k}{\fO_k^{\frac{1}{2}}}\right\|_F^2\right]+\frac{L\eta}{\delta(1-\theta)}\left(f(\X_1)-f^*\right)+\frac{K\eta(1-\theta)\sigma^2}{2\delta}\\
		\leq&f(\X_1)-f^*-\frac{\eta}{2}\sum_{k=1}^K\E_{\F_{k}}\left[\left\|\frac{\bP_k^\top\nabla f(\X_k)\bQ_k}{\fO_k^{\frac{1}{2}}}\right\|_F^2\right]+\frac{L\eta}{\delta(1-\theta)}\left(f(\X_1)-f^*\right)+\frac{K\eta(1-\theta)\hat\sigma^2}{2\delta},
	\end{aligned}
	\end{eqnarray}
	where we denote $\hat\sigma^2=\max\left\{\sigma^2,\frac{L\left(f(\X_1)-f^*\right)}{K\gamma^2}\right\}$ with any $\gamma\in(0,1]$. So we have
	\begin{eqnarray}
	\begin{aligned}\notag
		\sum_{k=1}^K\E_{\F_{k}}\left[\left\|\frac{\bP_k^\top\nabla f(\X_k)\bQ_k}{\fO_k^{\frac{1}{2}}}\right\|_F^2\right]
        \leq&\frac{2}{\eta}\left(f(\X_1)-f^*\right)+\frac{2L}{\delta(1-\theta)}\left(f(\X_1)-f^*\right)+\frac{K(1-\theta)\hat\sigma^2}{\delta}\\
		=&\frac{7}{\delta}\sqrt{K\hat\sigma^2L\left(f(\X_1)-f^*\right)}
	\end{aligned}
	\end{eqnarray}
	by letting $1-\theta=\sqrt{\frac{L\left(f(\X_1)-f^*\right)}{K\hat\sigma^2}}$ and $\eta=\sqrt{\frac{\delta^2\left(f(\X_1)-f^*\right)}{4LK\hat{\sigma}^2}}$, which also satisfies $\eta^2=\frac{\delta^2(1-\theta)^2}{4L^2}\leq\frac{\delta^2(1-\theta)^2}{2L^2}$ and $\eta\leq\frac{\delta}{2L}$. \hfill $\blacksquare$
\end{proof}

We present the proof of Lemma \ref{lemma-nuclear} below, which follows the technical framework of AdamW-style Shampoo established in \citep{shampoo-li-26}. The difference lies in that we introduce the control matrices $\mathbf{B}_{k}$ and $\mathbf{C}_{k}$ and converte $\nabla f(\X_{k})$ into $\frac{\bP_k^\top\nabla f(\X_k)\bQ_k}{\fO_k^{\frac{1}{2}}}$.
\begin{proof}[Proof of Lemma \ref{lemma-nuclear}]
Recall that we define
\begin{align}
	b_{k, i}\geq\max _{1 \leq j \leq n}\fO_{k,i,j},& \quad \mathbf{B}_{k}= \operatorname{diag}\left(b_{k, 1}, b_{k, 2}, \ldots b_{k, m}\right), \notag \\
	c_{k, j}\geq\max _{1 \leq i \leq m}\fO_{k,i,j},& \quad  \mathbf{C}_{k}= \operatorname{diag}\left(c_{k, 1}, c_{k, 2} \cdots c_{k, n}\right). \notag
\end{align}
So we have
\begin{eqnarray}
\begin{aligned}\notag
	&\left(\sum_{k=1}^{K} \E \left[ \left\| \nabla f(\X_{k})\right\|_{*}\right]\right)^{2}\overset{a}=\left(\sum_{k=1}^{K} \E \left[ \left\| \bP_k^\top \nabla f(\X_{k})\bQ_k\right\|_{*}\right]\right)^{2}\\
	\overset{b}\leq&\left(\sum_{k=1}^{K} \E\left[\left\|\mathbf{B}_{k}^{\frac{1}{4}}\right\|_{S_{4}}\left\|\mathbf{B}_{k}^{-\frac{1}{4}} \bP_k^\top \nabla f(\X_{k})\bQ_k \mathbf{C}_{k}^{-\frac{1}{4}}\right\|_{S_{2}}\left\|\mathbf{C}_{k}^{\frac{1}{4}}\right\|_{S_{4}}\right]\right)^{2} \\
	\overset{c}\leq&\left(\sum_{k=1}^{K} \E\left[\left\|\mathbf{B}_{k}^{\frac{1}{4}}\right\|_{S_{4}}^{4}\right]\right)^{\frac{1}{2}} \cdot \sum_{k=1}^{K} \E\left[\left\|\mathbf{B}_{k}^{-\frac{1}{4}} \bP_k^\top \nabla f(\X_{k})\bQ_k \mathbf{C}_{k}^{-\frac{1}{4}}\right\|_{S_{2}}^{2}\right] \cdot \left(\sum_{k=1}^{K} \E\left[\left\|\mathbf{C}_{k}^{\frac{1}{4}}\right\|_{S_{4}}^{4}\right]\right)^{\frac{1}{2}}\\
	\overset{d}=&\sqrt{\sum_{k=1}^{K} \E\left[\operatorname{tr} \mathbf{B}_{k}\right] \cdot \sum_{k=1}^{K} \E\left[\operatorname{tr} \mathbf{C}_{k}\right]} \cdot \sum_{k=1}^{K} \E\left[\left\|\mathbf{B}_{k}^{-\frac{1}{4}} \bP_k^\top \nabla f(\X_{k})\bQ_k \mathbf{C}_{k}^{-\frac{1}{4}}\right\|_F^{2}\right] \\
    =&\sqrt{\sum_{k=1}^{K} \E\left[\operatorname{tr} \mathbf{B}_{k}\right] \cdot \sum_{k=1}^{K} \E\left[\operatorname{tr} \mathbf{C}_{k}\right]} \cdot \sum_{k=1}^{K} \E\left[\sum_{i=1}^{m} \sum_{j=1}^{n}\left(\frac{\left(\bP_k^\top \nabla f(\X_{k})\bQ_k\right)_{i,j}}{\sqrt[4]{b_{k,i}c_{k,j}}}\right)^{2}\right] \\
	\overset{e}\leq& \frac{1}{2} \left(\sum_{k=1}^{K}\E\left[\operatorname{tr} \mathbf{B}_{k}\right]+\sum_{k=1}^{K}\E\left[\operatorname{tr} \mathbf{C}_{k}\right]\right) \cdot \sum_{k=1}^{K} \E\left[\sum_{i=1}^{m} \sum_{j=1}^{n}\left(\frac{\bP_k^\top \nabla f(\X_{k})\bQ_k}{\fO_k^{\frac{1}{2}}}\right)_{i,j}^{2}\right] \\
	=&\frac{1}{2} \left(\sum_{k=1}^{K}\E\left[\operatorname{tr}\mathbf{B}_{k}\right]+\sum_{k=1}^{K}\E\left[\operatorname{tr}\mathbf{C}_{k}\right]\right)\cdot \sum_{k=1}^K\E\left[\left\|\frac{\bP_k^\top\nabla f(\X_k)\bQ_k}{\fO_k^{\frac{1}{2}}}\right\|_F^2\right],
\end{aligned}
\end{eqnarray}
where we use the fact that $\|\cdot\|_*$ is a unitarily invariant norm such that $\|\bP^\top\A\bQ\|_*=\|\A\|_*$ for orthogonal matrices $\bP$ and $\bQ$ in $\overset{a}=$, the Schatten-H\"older inequality and $\|\X\|_*=\|\X\|_{S_1}$ in $\overset{b}\leq$,  H\"older's inequality in $\overset{c}\leq$, Lemma \ref{trace-transform} in $\overset{d}=$, and $b_{k,i}c_{k,j}\geq \fO_{k,i,j}^2$ in $\overset{e}\leq$.\hfill $\blacksquare$ 
\end{proof}

The following lemma is similar to the corresponding lemma in \citep{Li-2025-nips} and we list the proof here only for the sake of completeness. 
\begin{lemma}\label{momentum-bound-lemma}
	Suppose that Assumptions 1-3 and condition $\fO_{k,i,j}\geq \delta$ hold. Then for Algorithm \ref{general}, for $k\geq 2$, we have
	\begin{eqnarray}
	\begin{aligned}\label{GM-dif-equ}
		&\E_k\left[\left\|\nabla f(\X_k)-\M_k\right\|_F^2\big|\F_{k-1}\right]\\
		\leq&\theta\left\|\M_{k-1}-\nabla f(\X_{k-1})\right\|_F^2+\frac{L^2\eta^2}{\delta(1-\theta)}\left\|\frac{\bP_{k-1}^\top\M_{k-1}\bQ_{k-1}}{\fO_{k-1}^{\frac{1}{2}}}\right\|_F^2+(1-\theta)^2\sigma^2.
	\end{aligned}
	\end{eqnarray}
\end{lemma}
\begin{proof}
	Denoting $\Gamma_k=\G_k-\nabla f(\X_k)$, we have $\E_k\left[\Gamma_k\big|\F_{k-1}\right]=0$ and $\E_k\left[\left\|\Gamma_k\right\|_F^2\big|\F_{k-1}\right]\leq\sigma^2$. From the update of $\M_k$, we have
	\begin{eqnarray}
	\begin{aligned}\notag
		\M_k-\nabla f(\X_k)=& \theta \M_{k-1} + (1-\theta)\G_k - \nabla f(\X_k)\\
		=&\theta\left(\M_{k-1}-\nabla f(\X_{k-1})\right)+(1-\theta)\left(\nabla f(\X_k)+\Gamma_k\right)-\nabla f(\X_k)+\theta\nabla f(\X_{k-1})\\
		=& \theta \left(\M_{k-1}-\nabla f(\X_{k-1})\right) + (1-\theta)\Gamma_k - \theta\left(\nabla f(\X_k)-\nabla f(\X_{k-1})\right)
		\end{aligned}
	\end{eqnarray}
	and
	\begin{eqnarray}
	\begin{aligned}\notag
	   &\E_k\left[\left\|\nabla f(\X_k)-\M_k\right\|_F^2\big|\F_{k-1}\right]\\
	   =&\left\|\theta \left(\M_{k-1}-\nabla f(\X_{k-1})\right) - \theta\left(\nabla f(\X_k)-\nabla f(\X_{k-1})\right)\right\|_F^2 + (1-\theta)^2\E_k\left[\left\|\Gamma_k\right\|_F^2\big|\F_{k-1}\right]\\
	   \leq&\theta^2\left(1+\frac{1-\theta}{\theta}\right)\left\|\M_{k-1}-\nabla f(\X_{k-1})\right\|_F^2+\theta^2\left(1+\frac{\theta}{1-\theta}\right)\left\|\nabla f(\X_k)-\nabla f(\X_{k-1})\right\|_F^2\\
	   &+(1-\theta)^2\E_k\left[\left\|\Gamma_k\right\|_F^2\big|\F_{k-1}\right]
    \end{aligned}
	\end{eqnarray}
    \begin{eqnarray}
	\hspace*{-0.7cm}\begin{aligned}\notag
	   \leq&\theta\left\|\M_{k-1}-\nabla f(\X_{k-1})\right\|_F^2+\frac{1}{1-\theta}\left\|\nabla f(\X_k)-\nabla f(\X_{k-1})\right\|_F^2+(1-    \theta)^2\sigma^2\\
	   \leq&\theta\left\|\M_{k-1}-\nabla f(\X_{k-1})\right\|_F^2+\frac{L^2}{1-\theta}\left\|\X_k-\X_{k-1}\right\|_F^2+(1-\theta)^2\sigma^2\\
	   =&\theta\left\|\M_{k-1}-\nabla f(\X_{k-1})\right\|_F^2+\frac{L^2\eta^2}{1-\theta}\left\|\bP_{k-1}\frac{\bP_{k-1}^\top\M_{k-1}\bQ_{k-1}}{\fO_{k-1}}\bQ_{k-1}^\top\right\|_F^2+(1-\theta)^2\sigma^2\\
	   \overset{a}=&\theta\left\|\M_{k-1}-\nabla f(\X_{k-1})\right\|_F^2+\frac{L^2\eta^2}{1-\theta}\left\|\frac{\bP_{k-1}^\top\M_{k-1}\bQ_{k-1}}{\fO_{k-1}}\right\|_F^2+(1-\theta)^2\sigma^2\\
	   \overset{b}\leq&\theta\left\|\M_{k-1}-\nabla f(\X_{k-1})\right\|_F^2+\frac{L^2\eta^2}{\delta(1-\theta)}\left\|\frac{\bP_{k-1}^\top\M_{k-1}\bQ_{k-1}}{\fO_{k-1}^{\frac{1}{2}}}\right\|_F^2+(1-\theta)^2\sigma^2,
	\end{aligned}
	\end{eqnarray}
    where we use the fact that $\|\cdot\|_F$ is a unitarily invariant norm such that $\|\bP\A\bQ^\top\|_F=\|\A\|_F$ for orthogonal matrices $\bP$ and $\bQ$ in $\overset{a}=$, the notation that $\frac{\X}{\Y}$ and $\X^{\frac{1}{2}}$ are the elementwise division and square-root and condition $\fO_{k,i,j}\geq \delta$ in $\overset{b}\leq$.\hfill $\blacksquare$
\end{proof} 

\section{Proofs for SOAP}\label{app:proof-soap}
We now give the proof of Lemma \ref{trace-control-soap}. Instead of directly following the proof in \citep{Li-2025-nips}, which makes use of an $\F_{k-1}$-measurable upper bound on $\widetilde{\V}_k$, we separate the main recursion term from the stochastic error term. This demonstrates how the error term affects the main term in a more straightforward way, and also provides an explicit recursion for both terms.
\begin{proof}[Proof of Lemma \ref{trace-control-soap}]
	Consider $\mathbf{B}_k$. Denote
	\begin{align*}
		\phi_{k, i}&=\max _{1 \leq j \leq n}\left|\bP_k^\top \nabla f(\X_{k})\bQ_k\right|_{i, j},  \\
		\varphi_{k,i}&=\max _{1 \leq j \leq n}\left|\bP_k^\top \left(\G_{k}-\nabla f\left(\X_{k}\right)\right)\bQ_k\right|_{i, j}.
	\end{align*}
	In order to upper bound $\V_{k,i,j}$, define
	\begin{align*}
		\bar{\alpha}_{k, i}&=\beta \bar{\alpha}_{k-1, i}+2(1-\beta)\left(\phi_{k, i}^2+\varphi_{k,i}^2\right),\\
		b_{k,i}&=\sqrt{\bar{\alpha}_{k,i}+\varepsilon},
	\end{align*}
	with $\bar{\alpha}_{0, i}=0$. Therefore,
	\begin{align*}
		\left(\bP_k^\top\G_{k}\bQ_k\right)_{i, j}^2&=\left(\bP_k^\top \nabla f\left(\X_{k}\right)\bQ_k+\bP_k^\top\left(\G_{k}-\nabla f\left(\X_{k}\right)\right)\bQ_k\right)_{i, j}^2\\
		&\leq 2\left(\left(\bP_k^\top \nabla f(\X_{k})\bQ_k\right)_{i, j}^2+\left(\bP_k^\top \left(\G_{k}-\nabla f\left(\X_{k}\right)\right)\bQ_k\right)_{i, j}^2\right)\\
		&\leq 2\left(\phi_{k, i}^2+\varphi_{k,i}^2\right) .
	\end{align*}
	Compared to the definition of $\V_{k,i,j}$, by induction and $\V_{0,i,j}=\bar{\alpha}_{0,i}=0$, we obtain,
	\begin{equation}\label{soap-a-b-bound}
	\bar{\alpha}_{k,i}\geq \V_{k,i,j}\quad\mbox{and}\quad b_{k,i}\geq\sqrt{\V_{k,i,j}+\varepsilon}\quad\mbox{for all }j=1,2,\cdots,n,
	\end{equation}
	Thus, the coefficients $b_{k,i}$ satisfy (\ref{define-B-C}) since $\fO_k=\sqrt{\V_k+\varepsilon}$ for SOAP. Thus, $b_{k,i}$ is a valid choice for the trace-control process in the general framework. Then we exclude the stochastic error by recording the average error term:
	\begin{align*}
		\bar{\varphi}_{k, i}&=\beta \bar{\varphi}_{k-1, i}+2(1-\beta)\varphi_{k,i}^2,\\
		n_{k,i}&=\sqrt{\bar{\varphi}_{k,i}+\varepsilon},
	\end{align*}
	with $\bar{\varphi}_{0, i}=0$. Therefore,
	\begin{align*}
		b_{k,i}-n_{k,i}&=\frac{\bar{\alpha}_{k,i}-\bar{\varphi}_{k,i}}{b_{k,i}+n_{k,i}}\\
		&=\beta\frac{\bar{\alpha}_{k-1,i}-\bar{\varphi}_{k-1,i}}{b_{k,i}+n_{k,i}}+2(1-\beta)\frac{\phi_{k,i}^2}{b_{k,i}+n_{k,i}}\\
		&\overset{a}\leq\sqrt{\beta}\frac{\bar{\alpha}_{k-1,i}-\bar{\varphi}_{k-1,i}}{b_{k-1,i}+n_{k-1,i}}+2(1-\beta)\frac{\phi_{k,i}^2}{b_{k,i}}\\
		&=\sqrt{\beta}(b_{k-1,i}-n_{k-1,i})+2(1-\beta)\frac{\phi_{k,i}^2}{b_{k,i}},
	\end{align*}
	where we use $b_{k,i}^2=\bar{\alpha}_{k,i}+\varepsilon\geq \beta\bar{\alpha}_{k-1,i}+\beta\varepsilon=\beta b_{k-1,i}^2$ and a similar inequality for $n_{k,i}$ in $\overset{a}\leq$. Using this recursion and $b_{0,i}=n_{0,i}=\sqrt{\varepsilon}$, we have
	$$
	b_{k,i}\leq n_{k,i}+2(1-\beta) \sum_{t=1}^k \beta^{\frac{k-t}{2}}\frac{\phi_{t,i}^2}{b_{t,i}},
	$$
	and in conclusion,
	\begin{align*}
		\sum_{k=1}^K \E\left[\operatorname{tr}\mathbf{B}_k\right] &= \sum_{k=1}^K\sum_{i=1}^m \E\left[b_{k,i}\right]\\
		&\leq\sum_{k=1}^K\sum_{i=1}^m \E\left[n_{k,i}\right]+2(1-\beta)\sum_{k=1}^K\sum_{i=1}^m\sum_{t=1}^k \beta^{\frac{k-t}{2}}\E\left[\frac{\phi_{t,i}^2}{b_{t,i}}\right]\\
		&=\sum_{k=1}^K\sum_{i=1}^m\E\left[\sqrt{\bar{\varphi}_{k,i}+\varepsilon}\right]+2(1-\beta)\sum_{k=1}^K\sum_{i=1}^m\sum_{t=1}^k \beta^{\frac{k-t}{2}}\E\left[\frac{\phi_{t,i}^2}{b_{t,i}}\right]\\
		&\overset{b}\leq\sqrt{Km}\sqrt{\sum_{k=1}^K\sum_{i=1}^m\E\left[\bar{\varphi}_{k,i}\right]+Km\varepsilon}+2(1-\beta)\sum_{k=1}^K\sum_{i=1}^m\sum_{t=1}^k \beta^{\frac{k-t}{2}}\E\left[\frac{\phi_{t,i}^2}{b_{t,i}}\right]\\
		&=\sqrt{Km}\sqrt{2(1-\beta)\sum_{k=1}^K\sum_{i=1}^m\sum_{t=1}^k\beta^{k-t}\E\left[\varphi_{t,i}^2\right]+Km\varepsilon}+2(1-\beta)\sum_{k=1}^K\sum_{i=1}^m\sum_{t=1}^k \beta^{\frac{k-t}{2}}\E\left[\frac{\phi_{t,i}^2}{b_{t,i}}\right]\\
		&\overset{c}\leq\sqrt{Km}\sqrt{2\sum_{k=1}^K\sum_{i=1}^m\E\left[\varphi_{k,i}^2\right]+Km\varepsilon}+4\sum_{k=1}^K\sum_{i=1}^m\E\left[\frac{\phi_{k,i}^2}{b_{k,i}}\right]\\
		&\overset{d}\leq\sqrt{2Km}\sqrt{\sum_{k=1}^K\sum_{i=1}^m\E\left[\varphi_{k,i}^2\right]}+Km\sqrt{\varepsilon}+4\sum_{k=1}^K\sum_{i=1}^m\sum_{j=1}^n\E\left[\left(\frac{\bP_k^\top\nabla f(\X_{k})\bQ_k}{\sqrt[4]{\V_k+\varepsilon}}\right)_{i,j}^2\right]\\
		&\overset{e}\leq K\sqrt{2m}\sigma+Km\sqrt{\varepsilon}+4\sum_{k=1}^{K}\E\left[\left\|\frac{\bP_k^\top\nabla f(\X_{k})\bQ_k}{\sqrt[4]{\V_k+\varepsilon}}\right\|_F^2\right],
	\end{align*}
	where $\overset{b}\leq$ uses the concavity of $\sqrt{x}$, $\overset{c}\leq$ uses the fact that
	$$
	(1-\beta)\sum_{k=1}^K\sum_{i=1}^m\sum_{t=1}^k\beta^{k-t}\E\left[\varphi_{t,i}^2\right]
	=(1-\beta)\sum_{t=1}^K\sum_{i=1}^m\left(\E\left[\varphi_{t,i}^2\right]\sum_{k=t}^{K}\beta^{k-t}\right)
	\leq\sum_{k=1}^K\sum_{i=1}^m\E\left[\varphi_{k,i}^2\right],
	$$
	as well as,
	$$
	(1-\beta)\sum_{k=1}^K\sum_{i=1}^m\sum_{t=1}^k \beta^{\frac{k-t}{2}}\E\left[\frac{\phi_{t,i}^2}{b_{t,i}}\right]
	= (1-\beta)\sum_{t=1}^K\sum_{i=1}^m\left(\E\left[\frac{\phi_{t,i}^2}{b_{t,i}}\right]\sum_{k=t}^{K}\beta^{\frac{k-t}{2}}\right)
	\leq 2\sum_{k=1}^K\sum_{i=1}^m\E\left[\frac{\phi_{k,i}^2}{b_{k,i}}\right].
	$$
	In $\overset{d}\leq$, the first term uses $\sqrt{a+b}\leq \sqrt{a}+\sqrt{b}$, and the second term follows
	$$
	\frac{\phi_{k,i}^2}{b_{k,i}}\leq \sum_{j=1}^n\frac{\left(\bP_k^\top\nabla f(\X_{k})\bQ_k\right)_{i,j}^2}{b_{k,i}} \leq\sum_{j=1}^n\left(\frac{\bP_k^\top\nabla f(\X_{k})\bQ_k}{\sqrt[4]{\V_k+\varepsilon}}\right)_{i,j}^2,
	$$
    where we use (\ref{soap-a-b-bound}). In $\overset{e}\leq$, we use
    \begin{align}
    \sum_{k=1}^K\sum_{i=1}^m\E\left[\varphi_{k,i}^2\right]&\leq \sum_{k=1}^K\sum_{i=1}^m\sum_{j=1}^n\E\left[\left(\bP_k^\top\left(\G_k-\nabla f(\X_{k})\right)\bQ_k\right)_{i,j}^2\right]\notag\\
    &=\sum_{k=1}^K\E\left[\left\|\bP_k^\top\left(\G_k-\nabla f(\X_{k})\right)\bQ_k\right\|_F^2\right]=\sum_{k=1}^K\E\left[\left\|\G_k-\nabla f(\X_{k})\right\|_F^2\right]\leq K\sigma^2,\notag
    \end{align}
	where we use the fact that $\|\cdot\|_F$ is a unitarily invariant norm such that $\|\bP^\top\A\bQ\|_F=\|\A\|_F$ for orthogonal matrices $\bP$ and $\bQ$. Similarly, we obtain
	$$
	\sum_{k=1}^{K}\E\left[\operatorname{tr}\mathbf{C}_{k}\right]\leq K\left(\sqrt{2n}\sigma+n\sqrt{\varepsilon}\right)+4\sum_{k=1}^{K}\E\left[\left\|\frac{\bP_k^\top\nabla f(\X_{k})\bQ_k}{\sqrt[4]{\V_{k}+\varepsilon}}\right\|_{F}^{2}\right],
	$$
	and finish the proof by $\sigma\leq\hat\sigma$.\hfill $\blacksquare$
\end{proof}

The following gives the proof of Theorem \ref{SOAP}.
\begin{proof}[Proof of Theorem \ref{SOAP}]
    Plugging (\ref{soap-lemma}) in \eqref{nuclear}, we have
    \begin{align*}
        &\left(\sum_{k=1}^{K} \E \left[ \left\| \nabla f(\X_{k})\right\|_{*}\right]\right)^{2}\\
        \leq& \frac{1}{2} \left(\sum_{k=1}^{K}\E\left[\operatorname{tr} \mathbf{B}_{k}\right]+\sum_{k=1}^{K}\E\left[\operatorname{tr} \mathbf{C}_{k}\right]\right) \mathcal{S}_K\\ 
        \leq& \left(\frac{1}{2}K\left(\sqrt{2}\left(\sqrt{m}+\sqrt{n}\right)\hat{\sigma}+\left(m+n\right)\sqrt{\varepsilon}\right)+4\mathcal{S}_K\right)\cdot \mathcal{S}_K\\
        \leq& \left(K\sqrt{m+n}\left(\hat{\sigma}+\frac{1}{2}\sqrt{\left(m+n\right)\varepsilon}\right)+4\mathcal{S}_K\right)\cdot \mathcal{S}_K\\
        \overset{a}\leq& \left(2K\sqrt{m+n}\hat{\sigma}+4\mathcal{S}_K\right)\cdot \mathcal{S}_K\\
        \overset{b}\leq& \left(2K\sqrt{m+n}\hat{\sigma}+28\frac{\sqrt{K \hat{\sigma}^{2} L\left(f\left(\mathbf{X}_{1}\right)-f^{*}\right)}}{\delta}\right)\cdot 7 \frac{\sqrt{K \hat{\sigma}^{2} L\left(f\left(\mathbf{X}_{1}\right)-f^{*}\right)}}{\delta}\\
        =&14K\sqrt{m+n}\hat{\sigma}\cdot
        \frac{\sqrt{K \hat{\sigma}^{2} L\left(f\left(\mathbf{X}_{1}\right)-f^{*}\right)}}{\delta}+196 \frac{K \hat{\sigma}^{2} L\left(f\left(\mathbf{X}_{1}\right)-f^{*}\right)}{\delta^2},
    \end{align*}
    where we let $\varepsilon=\frac{\tau\hat{\sigma}^2}{m+n}$ with $\tau\leq 1$ in $\overset{a}\leq$, while using the upper bound for $\mathcal{S}_K$ in Lemma \ref{main1} at $\overset{b}\leq$. Therefore,
    \begin{align*}
        \frac{1}{K}\sum_{k=1}^{K} \mathbb{E}\left[\left\|\nabla f(\mathbf{X}_{k})\right\|_{*}\right]
        \leq&\sqrt{14}\sqrt[4]{\frac{(m+n)\hat{\sigma}^2}{\delta^2}}\sqrt[4]{\frac{\hat{\sigma}^{2}L\left(f(\mathbf{X}_{1})-f^{*}\right)}{K}}+14\sqrt{\frac{\hat{\sigma}^{2}L\left(f(\mathbf{X}_{1})-f^{*}\right)}{K\delta^2}}\\
        \overset{c}\leq& \left(4\left(\sqrt{m+n}\frac{\hat{\sigma}}{\delta}\right)^{\frac{1}{2}}+14\frac{\hat{\sigma}}{\delta}\right)\cdot \sqrt[4]{\frac{\hat{\sigma}^{2}L\left(f(\mathbf{X}_{1})-f^{*}\right)}{K}}\\
        \leq& \left(2\sqrt{m+n}+16\frac{\hat{\sigma}}{\delta}\right)\cdot \sqrt[4]{\frac{\hat{\sigma}^{2}L\left(f(\mathbf{X}_{1})-f^{*}\right)}{K}},
    \end{align*}
    where we use $\hat\sigma^4\geq \frac{\hat{\sigma}^{2}L\left(f(\mathbf{X}_{1})-f^{*}\right)}{K}$ from the definition of $\hat\sigma$ in $\overset{c}\leq$.\hfill $\blacksquare$
\end{proof}

\section{Proofs for SPlus}\label{appendix-c}
The SPlus algorithm proposed in \citep{splus-2025-nips} is shown in Algorithm \ref{splus}. We extract the core idea of sign-based normalization, replace the sign-type update by a truncated modification, absorb the factor $\frac{2}{m+n}$ into $\eta$, and omit the iterate averaging, which leads to a suitable generalized Algorithm \ref{gen-splus}. 
We take a slightly different analysis approach for SPlus: since the operation $\bPh_k$ in this algorithm is computed directly from the current rotated momentum (unlike SOAP, which includes a historical averaging term $\V_k$), the entire analysis is simplified. We can deduce the following lemma from Lemma \ref{momentum-bound-lemma}. It is crucial for non-historic dependent update algorithms, such as SPlus.

\begin{lemma}\label{error-bounding}
	Suppose that Assumptions 1-3 and condition $\fO_{k,i,j}\geq\delta$ hold. Let $\hat\sigma^2=\max\left\{\sigma^2,\frac{L\left(f(\X_1)-f^*\right)}{K\gamma^2}\right\}$ with any $\gamma\in(0,1]$, $1-\theta=\sqrt{\frac{L\left(f(\X_1)-f^*\right)}{K\hat\sigma^2}}$, and $\eta=\sqrt{\frac{\delta^2\left(f(\X_1)-f^*\right)}{4LK\hat\sigma^2}}$. Then for the general Algorithm \ref{general}, we have
	$$
	\sum_{k=1}^K\E_{\F_k}\left[\left\|\M_k-\nabla f(\X_k) \right\|_F\right]\leq \sqrt{13}K\hat{\sigma}.
	$$
\end{lemma}
\begin{proof}
	Taking expectation on both sides of (\ref{GM-dif-equ}) with respect to $\zeta_k$ conditioned on $\F_{k-1}$, summing over $k=2,3,\cdots,K$, and adding (\ref{variancebound}), we have
	\begin{align*}
		&\sum_{k=1}^K\E_{\F_k}\left[\left\|\nabla f(\X_k)-\M_k\right\|_F^2\right]\\
		\leq&\theta\sum_{k=2}^K\E_{\F_{k-1}}\left[\left\|\M_{k-1}-\nabla f(\X_{k-1})\right\|_F^2\right]+\frac{L^2\eta^2}{\delta(1-\theta)}\sum_{k=2}^K\E_{\F_{k-1}}\left[\left\|\frac{\bP_{k-1}^\top\M_{k-1}\bQ_{k-1}}{\fO_{k-1}^{\frac{1}{2}}}\right\|_F^2\right]\\
        &+(K-1)(1-\theta)^2\sigma^2+2L\left(f(\X_1)-f^*\right)+(1-\theta)^2\sigma^2\\
        \leq&\theta\sum_{k=1}^K\E_{\F_k}\left[\left\|\M_k-\nabla f(\X_k)\right\|_F^2\right]+\frac{L^2\eta^2}{\delta(1-\theta)}\sum_{k=1}^K\E_{\F_k}\left[\left\|\frac{\bP_k^\top\M_k\bQ_k}{\fO_k^{\frac{1}{2}}}\right\|_F^2\right]\\
        &+K(1-\theta)^2\sigma^2+2L\left(f(\X_1)-f^*\right)\\
        \leq&\theta\sum_{k=1}^K\E_{\F_k}\left[\left\|\M_k-\nabla f(\X_k)\right\|_F^2\right]+\frac{2L^2\eta^2}{\delta(1-\theta)}\sum_{k=1}^K\E_{\F_k}\left[\left\|\frac{\bP_k^\top\nabla f(\X_k)\bQ_k}{\fO_k^{\frac{1}{2}}}\right\|_F^2\right]\\
        &+\frac{2L^2\eta^2}{\delta(1-\theta)}\sum_{k=1}^K\E_{\F_k}\left[\left\|\frac{\bP_k^\top(\M_k-\nabla f(\X_k))\bQ_k}{\fO_k^{\frac{1}{2}}}\right\|_F^2\right]+K(1-\theta)^2\sigma^2+2L\left(f(\X_1)-f^*\right)\\
        \overset{a}\leq&\theta\sum_{k=1}^K\E_{\F_k}\left[\left\|\M_k-\nabla f(\X_k)\right\|_F^2\right]+\frac{2L^2\eta^2}{\delta(1-\theta)}\sum_{k=1}^K\E_{\F_k}\left[\left\|\frac{\bP_k^\top\nabla f(\X_k)\bQ_k}{\fO_k^{\frac{1}{2}}}\right\|_F^2\right]\\
        &+\frac{2L^2\eta^2}{\delta^2(1-\theta)}\sum_{k=1}^K\E_{\F_k}\left[\left\|\bP_k^\top(\M_k-\nabla f(\X_k))\bQ_k\right\|_F^2\right]+K(1-\theta)^2\sigma^2+2L\left(f(\X_1)-f^*\right)\\
        \overset{b}\leq&\theta\sum_{k=1}^K\E_{\F_k}\left[\left\|\M_k-\nabla f(\X_k)\right\|_F^2\right]+\frac{\delta(1-\theta)}{2}\sum_{k=1}^K\E_{\F_k}\left[\left\|\frac{\bP_k^\top\nabla f(\X_k)\bQ_k}{\fO_k^{\frac{1}{2}}}\right\|_F^2\right]\\
        &+\frac{1-\theta}{2}\sum_{k=1}^K\E_{\F_k}\left[\left\|\M_k-\nabla f(\X_k)\right\|_F^2\right]+K(1-\theta)^2\sigma^2+2L\left(f(\X_1)-f^*\right),
    \end{align*}
    where we use $\fO_{k,i,j}\geq\delta$ in $\overset{a}\leq$, $\frac{L^2\eta^2}{\delta^2(1-\theta)^2}\leq \frac{1}{4}$ and the fact that $\|\cdot\|_F$ is a unitarily invariant norm in $\overset{b}\leq$. So we have
	\begin{align*}
		&\sum_{k=1}^K\E_{\F_k}\left[\left\|\nabla f(\X_k)-\M_k\right\|_F^2\right]\\
        \leq& \frac{4L\left(f(\X_1)-f^*\right)}{1-\theta}+2K(1-\theta)\sigma^2+\delta\sum_{k=1}^K\E_{\F_k}\left[\left\|\frac{\bP_k^\top\nabla f(\X_k)\bQ_k}{\fO_k^{\frac{1}{2}}}\right\|_F^2\right]\\
        \overset{c}\leq& \frac{4L\left(f(\X_1)-f^*\right)}{1-\theta}+2K(1-\theta)\sigma^2+7\sqrt{K\hat\sigma^2L\left(f(\X_1)-f^*\right)}\\\overset{d}\leq&13\sqrt{K\hat\sigma^2L\left(f(\X_1)-f^*\right)}\overset{e}\leq 13K\hat\sigma^2,
	\end{align*}
	where we use Lemma \ref{main1} in $\overset{c}\leq$,  $1-\theta=\sqrt{\frac{L\left(f(\X_1)-f^*\right)}{K\hat\sigma^2}}$ and $\sigma^2\leq\hat\sigma^2$ in $\overset{d}\leq$, and $L\left(f(\X_1)-f^*\right)\leq K\hat\sigma^2$ from the definition of $\hat\sigma$ in $\overset{e}\leq$. Using Cauchy-Schwarz inequality, we have
	\begin{align*}
		\sum_{k=1}^K\E_{\F_k}\left[\left\|\nabla f(\X_k)-\M_k\right\|_F\right]\leq& \sqrt{K\sum_{k=1}^K\E_{\F_k}\left[\left\|\nabla f(\X_k)-\M_k\right\|_F^2\right]}\leq \sqrt{13}K\hat\sigma.
	\end{align*}
    \hfill $\blacksquare$
\end{proof}

The following lemma is an analogue of Lemma \ref{trace-control-soap}, but with a simpler proof.
\begin{lemma}\label{trace-control-splus}
	Suppose that Assumptions 1-3 and condition $\fO_{k,i,j}\geq\delta$ hold. Let $\hat\sigma^2=\max\left\{\sigma^2,\frac{L\left(f(\X_1)-f^*\right)}{K\gamma^2}\right\}$ with any $\gamma\in(0,1]$, $1-\theta=\sqrt{\frac{L\left(f(\X_1)-f^*\right)}{K\hat\sigma^2}}$, and $\eta=\sqrt{\frac{\delta^2\left(f(\X_1)-f^*\right)}{4LK\hat\sigma^2}}$. 
	Then for Algorithm \ref{gen-splus}, we have
	\begin{eqnarray}
    \begin{aligned}\label{bound-splus}
		\sum_{k=1}^{K}\E\left[\operatorname{tr} \mathbf{B}_{k}\right]&\leq 2\sqrt{13m}K\hat{\sigma}+2Km\delta+\sum_{k=1}^K\E_{\F_{k}}\left[\left\|\frac{\bP_k^\top\nabla f(\X_k)\bQ_k}{\fO_k^{\frac{1}{2}}}\right\|_F^2\right],\\
		\sum_{k=1}^{K}\E\left[\operatorname{tr} \mathbf{C}_{k}\right]& \leq 2\sqrt{13n}K\hat{\sigma}+2Kn\delta+\sum_{k=1}^K\E_{\F_{k}}\left[\left\|\frac{\bP_k^\top\nabla f(\X_k)\bQ_k}{\fO_k^{\frac{1}{2}}}\right\|_F^2\right].
	\end{aligned}
    \end{eqnarray}
\end{lemma}
\begin{proof}\label{form}
Different from the definition for SOAP, we instead define
$$
b_{k,i}=\max_{1 \leq j \leq n}\left|\bP_k^\top\M_k\bQ_k\right|_{i,j}+\delta\geq \max_{1 \leq j \leq n}\fO_{k,i,j}
$$
for (\ref{SPlus-def}). Denote
$$
j_i^*=\argmax_{1\leq j \leq n}\left|\bP_k^\top\M_k\bQ_k\right|_{i,j}.
$$
Therefore,
$$
b_{k,i}=\left|\bP_k^\top\M_k\bQ_k\right|_{i,j_i^*}+\delta\leq \left|\bP_k^\top\nabla f(\X_k)\bQ_k\right|_{i,j_i^*}+\left|\bP_k^\top\left(\M_k-\nabla f(\X_k) \right)\bQ_k\right|_{i,j_i^*}+\delta.
$$
Since
$$
\left|\bP_k^\top\nabla f(\X_k)\bQ_k\right|_{i,j_i^*}\leq \frac{\left|\bP_k^\top\nabla f(\X_k)\bQ_k\right|_{i,j_i^*}^2}{2b_{k,i}}+\frac{b_{k,i}}{2},
$$
plugging in gives
\begin{align*}
	b_{k,i}&\leq\frac{\left|\bP_k^\top\nabla f(\X_k)\bQ_k\right|_{i,j_i^*}^2}{b_{k,i}}+2\left|\bP_k^\top\left(\M_k-\nabla f(\X_k) \right)\bQ_k\right|_{i,j_i^*}+2\delta\\
    &\leq\sum_{j=1}^n\frac{\left|\bP_k^\top\nabla f(\X_k)\bQ_k\right|_{i,j}^2}{b_{k,i}}+2\left\|\left(\bP_k^\top\left(\M_k-\nabla f(\X_k) \right)\bQ_k\right)_{i,:}\right\|_2+2\delta\\
	&\leq\sum_{j=1}^n\left(\frac{\bP_k^\top\nabla f(\X_k)\bQ_k}{\fO_k^{\frac{1}{2}}}\right)_{i,j}^2+2\left\|\left(\bP_k^\top\left(\M_k-\nabla f(\X_k) \right)\bQ_k\right)_{i,:}\right\|_2+2\delta.
\end{align*}
Summing over $i=1,\ldots,m$ and $k=1,\ldots,K$ gives
\begin{eqnarray}
\begin{aligned}\notag
	&\sum_{k=1}^K\E\left[\operatorname{tr} \mathbf{B}_{k}\right]=\sum_{k=1}^K\sum_{i=1}^m \E\left[b_{k,i}\right]\\
	\leq&\sum_{k=1}^K\hspace*{-0.05cm}\sum_{i=1}^m\hspace*{-0.05cm}\sum_{j=1}^n\hspace*{-0.02cm}\E_{\F_k}\hspace*{-0.13cm}\left[\hspace*{-0.07cm}\left(\hspace*{-0.05cm}\frac{\bP_k^\top\nabla f(\X_k)\bQ_k}{\fO_k^{\frac{1}{2}}}\hspace*{-0.04cm}\right)_{i,j}^2\hspace*{-0.02cm}\right]\hspace*{-0.09cm}+\hspace*{-0.07cm}2\hspace*{-0.02cm}\sum_{k=1}^K\hspace*{-0.05cm}\sum_{i=1}^m\hspace*{-0.02cm}\E_{\F_k}\hspace*{-0.11cm}\left[\left\|\left(\bP_k^\top\hspace*{-0.05cm}\left(\M_k\hspace*{-0.07cm}-\hspace*{-0.07cm}\nabla f(\X_k) \right)\bQ_k\right)_{i,:}\right\|_2\right]\hspace*{-0.09cm}+\hspace*{-0.07cm}2Km\delta\\
	\leq&\sum_{k=1}^K\E_{\F_{k}}\hspace*{-0.11cm}\left[\hspace*{-0.02cm}\left\|\frac{\bP_k^\top\nabla f(\X_k)\bQ_k}{\fO_k^{\frac{1}{2}}}\right\|_F^2\right]\hspace*{-0.09cm}+\hspace*{-0.07cm}2\hspace*{-0.02cm}\sum_{k=1}^K\hspace*{-0.02cm}\E_{\F_k}\hspace*{-0.11cm}\left[\hspace*{-0.02cm}\sqrt{m\sum_{i=1}^m\left\|\left(\bP_k^\top\left(\M_k\hspace*{-0.09cm}-\hspace*{-0.09cm}\nabla f(\X_k) \right)\bQ_k\right)_{i,:} \right\|_2^2}\right]\hspace*{-0.09cm}+\hspace*{-0.07cm}2Km\delta\\ 
    =&\sum_{k=1}^K\E_{\F_{k}}\left[\left\|\frac{\bP_k^\top\nabla f(\X_k)\bQ_k}{\fO_k^{\frac{1}{2}}}\right\|_F^2\right]+2\sqrt{m}\sum_{k=1}^K\E_{\F_k}\left[\left\|\bP_k^\top\left(\M_k-\nabla f(\X_k) \right)\bQ_k\right\|_F\right]+2Km\delta\\
	=&\sum_{k=1}^K\E_{\F_{k}}\left[\left\|\frac{\bP_k^\top\nabla f(\X_k)\bQ_k}{\fO_k^{\frac{1}{2}}}\right\|_F^2\right]+2\sqrt{m}\sum_{k=1}^K\E_{\F_k}\left[\left\|\M_k-\nabla f(\X_k) \right\|_{F}\right]+2Km\delta\\
	\leq& \sum_{k=1}^K\E_{\F_{k}}\left[\left\|\frac{\bP_k^\top\nabla f(\X_k)\bQ_k}{\fO_k^{\frac{1}{2}}}\right\|_F^2\right]+2\sqrt{13m}K\hat{\sigma}+2Km\delta,
\end{aligned}
\end{eqnarray}
where we use Lemma \ref{error-bounding} in the last inequality. Similarly,
$$
\sum_{k=1}^K\E\left[\operatorname{tr} \mathbf{C}_{k}\right]\leq \sum_{k=1}^K\E_{\F_{k}}\left[\left\|\frac{\bP_k^\top\nabla f(\X_k)\bQ_k}{\fO_k^{\frac{1}{2}}}\right\|_F^2\right]+2\sqrt{13n}K\hat{\sigma}+2Kn\delta.
$$
\hfill $\blacksquare$
\end{proof}
Similar to the proof of Theorem \ref{SOAP}, the following gives the proof of Theorem \ref{SPlusCon}.
\begin{proof}[Proof of Theorem \ref{SPlusCon}]
Substituting (\ref{bound-splus}) into \eqref{nuclear} and using the bound of $\mathcal{S}_K$ in Lemma \ref{main1}, we have
\begin{align*}
        &\left(\sum_{k=1}^{K} \E \left[ \left\| \nabla f(\X_{k})\right\|_{*}\right]\right)^{2}\\
        \leq& \left(\sqrt{13}\left(\sqrt{m}+\sqrt{n}\right)K\hat{\sigma}+K(m+n)\delta+\mathcal{S}_K\right)\cdot \mathcal{S}_K\\
        \leq& \left(\sqrt{38(m+n)}K\hat{\sigma}+7\frac{\sqrt{K \hat{\sigma}^{2} L\left(f\left(\mathbf{X}_{1}\right)-f^{*}\right)}}{\delta}\right)\cdot 7 \frac{\sqrt{K \hat{\sigma}^{2} L\left(f\left(\mathbf{X}_{1}\right)-f^{*}\right)}}{\delta}\\
        =&7\sqrt{38(m+n)}K\hat{\sigma}\cdot
        \frac{\sqrt{K \hat{\sigma}^{2} L\left(f\left(\mathbf{X}_{1}\right)-f^{*}\right)}}{\delta}+49 \frac{K \hat{\sigma}^{2} L\left(f\left(\mathbf{X}_{1}\right)-f^{*}\right)}{\delta^2},
    \end{align*}
    and
    \begin{align*}
        \frac{1}{K}\sum_{k=1}^{K} \mathbb{E}\left[\left\|\nabla f(\mathbf{X}_{k})\right\|_{*}\right]
        \leq&7\sqrt[4]{\frac{(m+n)\hat{\sigma}^2}{\delta^2}}\sqrt[4]{\frac{\hat{\sigma}^{2}L\left(f(\mathbf{X}_{1})-f^{*}\right)}{K}}+7\sqrt{\frac{\hat{\sigma}^{2}L\left(f(\mathbf{X}_{1})-f^{*}\right)}{K\delta^2}}\\
        \leq& \left(7\left(\sqrt{m+n}\frac{\hat{\sigma}}{\delta}\right)^{\frac{1}{2}}+7\frac{\hat{\sigma}}{\delta}\right)\cdot \sqrt[4]{\frac{\hat{\sigma}^{2}L\left(f(\mathbf{X}_{1})-f^{*}\right)}{K}}\\
        \leq& \left(4\sqrt{m+n}+11\frac{\hat{\sigma}}{\delta}\right)\cdot \sqrt[4]{\frac{\hat{\sigma}^{2}L\left(f(\mathbf{X}_{1})-f^{*}\right)}{K}},
    \end{align*}
    where we use $\delta=\frac{\tau\hat{\sigma}}{\sqrt{m+n}}$. Substituting this setting yields the stated bound.
    \hfill $\blacksquare$
\end{proof}

\end{document}